\documentclass[11pt]{article}
\usepackage{amsmath}
\usepackage{amstext,amsbsy}
\usepackage{epsfig,multicol}
\usepackage{amsthm}
\usepackage{amssymb,latexsym}
\usepackage{color}
\usepackage{tikz}
\usepackage{subfigure}
\usepackage[vcentermath,enableskew]{youngtab}
\usepackage{authblk}

\input pathlate.sty

\newtheorem{thm}{Theorem}

\newtheorem{exa}[thm]{Example}
\newtheorem{lem}[thm]{Lemma}
\newtheorem{cor}[thm]{Corollary}
\newtheorem{prop}[thm]{Proposition}
\newtheorem{conj}[thm]{Conjecture}

\def \se {\mathrm{se}} 
\def \ne {\mathrm{ne}}  

\begin{document}
\title{Hecke insertion and maximal increasing and decreasing sequences in fillings of stack polyominoes}
\author[$\ast$,$\dagger$]{Ting Guo}
\author[$\ddag$]{Svetlana Poznanovi\'c}
\affil[$\ast$]{MOE-LCSM, School of Mathematics and Statistics, Hunan Normal University, China.}
\affil[$\dagger$]{Faculty of Mathematics, University of Vienna, Austria.}
\affil[$\ddag$]{School of Mathematical and Statistical Sciences, Clemson University, USA.}

\date{} 
\maketitle
\begin{abstract}
We prove that the number of $01$-fillings of a given stack polyomino (a polyomino with justified rows whose lengths form a unimodal sequence) with at most one 1 per column which do not contain a fixed-size northeast chain and a fixed-size southeast chain, depends only on the set of row lengths of the polyomino. The proof is via a bijection between fillings of stack polyominoes which differ only in the position of one row and uses the Hecke insertion algorithm by Buch, Kresch, Shimozono, Tamvakis, and Yong and the jeu de taquin for  increasing tableaux of Thomas and Yong. Moreover, our bijection gives another proof of the result by Chen, Guo, and Pang that the crossing number and the nesting number have a symmetric joint distribution over linked partitions.

\end{abstract}



{\renewcommand{\thefootnote}{} \footnote{\emph{E-mail addresses}:
ting.guo@univie.ac.at (T.~Guo), spoznan@clemson.edu (S.~Poznanovi\'c)}

\footnotetext[1]{T.~Guo was supported by the Austrian Science Fund FWF grant SFB F50, the China Scholarship Council, and the Construct Program of the Key Discipline in Hunan Province. S.~Poznanovi\'c was supported by {NSF-DMS} 1815832.} 


\section{Introduction} \label{S:introduction}


A \emph{polyomino } is a finite subset of $\mathbb{Z}^2$ 
where every element of $\mathbb{Z}^2$ is represented by a square box. 
The polyomino is \emph{row-convex} (resp. \emph{column-convex}) if its every row (resp. column) is connected.
If the polyomino is both row- and column-convex, we say that it is {\it convex}.
It is \emph{intersection-free\/} if every two columns are \emph{comparable\/}, 
i.e., the row-coordinates of one column form a subset of those of the other column. 
Equivalently, it is intersection-free if every two rows are comparable.
A \emph{moon polyomino\/} is a convex and intersection-free polyomino (e.g. Figure~\ref{fig:1}a). 
By \emph{stack polyomino\/} in this paper, we mean a moon polyomino in which 
the columns are arranged by length in descending order from left to right, so that the rows are left justified (e.g. Figure~\ref{fig:1}b). 
We note that the term 'stack polyomino' has been used in other papers to denote polyominoes with justified columns rather than rows.

\begin{figure}[h]
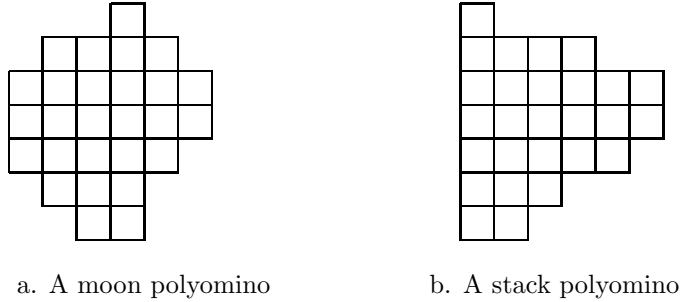

$$
\Einheit.15cm
\PfadDicke{.5pt}
\Pfad(0,6),222222222\endPfad
\Pfad(3,3),222222222222222\endPfad
\Pfad(6,0),222222222222222222\endPfad
\Pfad(9,0),222222222222222222222\endPfad
\Pfad(12,0),222222222222222222222\endPfad
\Pfad(15,6),222222222222\endPfad
\Pfad(18,9),222222\endPfad
\Pfad(9,21),111\endPfad
\Pfad(3,18),111111111111\endPfad
\Pfad(0,15),111111111111111111\endPfad
\Pfad(0,12),111111111111111111\endPfad
\Pfad(0,9),111111111111111111\endPfad
\Pfad(0,6),111111111111111\endPfad
\Pfad(3,3),111111111\endPfad
\Pfad(6,0),111111\endPfad
\hbox{\hskip6cm}
\PfadDicke{.5pt}
\Pfad(0,0),111111\endPfad
\Pfad(0,3),111111111\endPfad
\Pfad(0,6),111111111111111\endPfad
\Pfad(0,9),111111111111111111\endPfad
\Pfad(0,12),111111111111111111\endPfad
\Pfad(0,15),111111111111111111\endPfad
\Pfad(0,18),111111111111\endPfad
\Pfad(0,21),111\endPfad
\Pfad(0,0),222222222222222222222\endPfad
\Pfad(3,0),222222222222222222222\endPfad
\Pfad(6,0),222222222222222222\endPfad
\Pfad(9,3),222222222222222\endPfad
\Pfad(12,6),222222222222\endPfad
\Pfad(15,6),222222222\endPfad
\Pfad(18,9),222222\endPfad
\hskip3cm
$$
\centerline{\small a. A moon polyomino
\hskip2cm
b. A stack polyomino}
\caption{A moon polyomino and a stack polyomino with same row lengths.}
\label{fig:1}
\end{figure}

The main result in this paper is about certain fillings of stack polyominoes. An \emph{arbitrary filling} of a polyomino is an assignment of natural numbers to the boxes of the polyomino.  Here we are concerned with $01$-fillings of moon polyomino with restricted column sums. 
That is, given a moon polyomino $\mathcal{M}$,  a $0$ or a $1$ is assigned to each box of $\mathcal{M}$ so that 
there is at most one $1$ in each column.  We simply use the term \emph{filling} to denote such a $01$-filling. 
A box is empty if it is assigned a $0$ and it is a $1$-box otherwise.  A \emph{chain} is a sequence of non-zero entries in a filling such that the smallest rectangle containing all the elements of the sequence is completely contained in the moon polyomino. 
The \emph{length} of a chain is the number of entries in the chain. A \emph{ne-chain} of a filling is a sequence of non-zero entries in the fillings, 
such that each entry is strictly to the right and strictly above the preceding entry of the sequence. 
Similarly,  an \emph{se-chain} of a filling is a sequence of non-zero entries in the filling such that each entry is strictly to the right and strictly below the preceding entry of the sequence.   Let $N(\mathcal{M};n;ne=u,se=v)$ be the number of $01$-fillings of the moon polyomino $\mathcal{M}$ with at most one $1$ in each column such that: 
the sum of entries equal to $n$, the longest $ne$-chain has length $u$, and the longest $se$-chain has length $v$. Our main result is the following.

\begin{thm} \label{Thm1} 
Let $\mathcal{M}$ be a stack polyomino and $\sigma$ be a permutation of the rows of $\mathcal{M}$, 
such that $\sigma \mathcal{M}$ is a stack polyomino. Then
\[N(\mathcal{M};n;ne=u,se=v)=N(\sigma\mathcal{M};n;ne=u,se=v).\]
\end{thm}

In other words, the number $N(\mathcal{M};n;ne=u,se=v)$ only depends on the row lengths of the stack polyomino $\mathcal{M}$. Since the reflection through the horizontal axis swaps \emph{se}- and \emph{ne}-chains, this implies that for each stack polyomino $\mathcal{M}$ the joint distribution of $\mathrm{ne}$ and $\mathrm{se}$ is symmetric. Precisely, we have the following.

\begin{cor} \label{Cor1} 
Let $\mathcal{M}$ be a stack polyomino . Then
\[N(\mathcal{M};n;ne=u,se=v)=N(\mathcal{M};n;ne=v,se=u).\]
\end{cor}

Chains in fillings of polyominoes have been studied by many people before. Using the Fomin's growth diagrams, Krattenthaler~\cite{Krattenthaler}  proved the analogue of our result  for 01-fillings of Ferrers diagrams with at most one 1 in each row \emph{and} column as well as variants for $\mathbb{N}$-fillings. In fact, restricting the number of 1's in the rows in addition to the columns guarantees that the result holds even if we preserve the row sums, which is not the case in the present paper (see Section~\ref{S:discussion} for details). If one considers 01-fillings with arbitrary row and column sums, then it is known that the number of fillings in which $ne=u$ depends only on the set of columns of the polyomino, but not the shape of the polyomino.  Algebraic proofs of this fact  for stack polyominoes are due to Jonsson~\cite{Jon05} and Jonsson and Welker~\cite{JW07}. Rubey~\cite{Rubey} extended this result to moon polyominoes using a bijection based on the Robinson-Schenstead-Knuth (RSK) insertion and jeu de taquin as well as the inclusion-exclusion principle. A completely  bijective proof and an extension to almost-moon polyominoes (in which the convexity  property is broken by only one row) was given by Poznanovi\'c and Yan~\cite{PY}. 

In the context of fillings of Ferrers shapes, the results mentioned above can be viewed as extensions of results about enumeration of:  permutations and involutions with restricted patterns of Backelin, West, and Xin~\cite{BWX} and of Bousquet-M\'elou and Steingr{\'i}msson~\cite{BMS}, crossings and nestings in matchings and set partitions due to Chen, Deng, Du, Stanley, and Yan~\cite{CDDSY} or graphs with restricted right and left degree sequences as shown by de Mier~\cite{deMier}.

As we discuss in Section~\ref{S:FS}, our result is an extension of the result of Chen, Guo, and Pang~\cite{CGP} about the symmetry of the crossing and the nesting numbers over linked partitions. To prove their result, Chen, Guo, and Pang introduced vacillating Hecke tableaux and found a bijection between vacillating Hecke tableaux and linked partitions using the Hecke insertion of Buch, Kresch, Shimozono, Tamvakis, and Yong~\cite{BKSTY}. While our proof is based on a bijection using the Hecke insertion as well, our bijection is different from the one in ~\cite{CGP} even for the case of linked partitions.

This paper is organized as follows. To prove Theorem~\ref{Thm1}, we construct a bijection between fillings of $\mathcal{M}$ and fillings of $\mathcal{\sigma M}$ where $\sigma$ only moves the bottom row of $\mathcal{M}$. The bijection, described in Section~\ref{S:proof} uses the Hecke insertion of Buch, Kresch, Shimozono, Tamvakis and Yong~ \cite{BKSTY} and the jeu de taquin for increasing tableaux of Thomas, and Yong~ \cite{TY09}. In addition, the proof uses $K$-Knuth equivalence on words and increasing tableaux~\cite{BuchSamuel, TY11}. The background that we need is presented in Section~\ref{S:tools}. In Section~\ref{S:FS}, we show that the main result in~\cite{CGP} is a special case of Theorem~\ref{Thm1} and give an example to show that our bijection for the special case is different from the one in~\cite{CGP}. Finally, in Section~\ref{S:discussion} we discuss possible and impossible extensions of Theorem~\ref{Thm1}.

\section{Main Tools}\label{S:tools}


The Hecke insertion~\cite{BKSTY}, the Hecke growth diagrams~\cite{PP}, the jeu de taquin for increasing tableaux~\cite{TY09}, 
and the notion of $K$-Knuth equivalence on words and increasing tableaux \cite{BuchSamuel, TY11} are the main tools in Section~\ref{S:proof}. For the convenience of the reader, in this section we give all the necessary definitions, following the original description and notation as much as possible. The description of all of these concepts except the Hecke growth diagrams can also be found in~\cite{GMPPRST}.

\subsection{Hecke Insertion}

The Hecke insertion was developed in~\cite{BKSTY} for the study of the stable Grothendieck polynomials.

An {\it increasing tableau} is a filling of a Ferrers shape assigning to each box a positive integer
such that the entries strictly increase along each row and down each column.
The Hecke insertion is a procedure to insert a positive integer $x$ into an increasing tableau $Y$,
resulting in another increasing  tableau $Z$, so that either $Z$ has the same shape as $Y$ 
or it has one extra box $c$.
In the case when $Z$ has the same shape as $Y$, it also contains a special corner $c$ where the insertion algorithm terminated. A parameter $\alpha\in \lbrace 0, 1 \rbrace$ is used to distinguish these two cases: $\alpha$ is set to $1$ if and only if the corner $c$ is outside the shape of $Y$. 
Thus the complete output of the insertion algorithm is the triple $(Z, c, \alpha)$.
We write $(Z, c, \alpha)= (Y\xleftarrow{\,\mathrm{H}}x)$.

The Hecke insertion algorithm proceeds by inserting the integer $x$ into the first row of $Y$. 
This may modify this row, and possibly produce an {\it output integer}, 
which is then inserted into the second row of $Y$, etc. 
This process is repeated until an insertion does not produce an output integer. 
The rules for inserting an integer $x$ into a row $R$ are as follows. 

If $x$ is larger than or equal to all entries in $R$, then no output integer is produced and the algorithm terminates. If adding $x$ as a new box to the end of $R$ results in an increasing tableau, 
then $Z$ is the resulting tableau, $\alpha=1$, and $c$ is the corner where $x$ was added.
If adding $x$ as a new box to the end of $R$ does not result in an increasing tableau, 
then let $Z=Y$, $\alpha=0$, and $c$ is the corner at the bottom of the column of $Z$ 
containing the rightmost box of $R$.

Otherwise the integer $x$ is strictly smaller than some entry in $R$, 
and we let $y$ be the smallest integer in $R$ that is strictly larger than $x$.
If replacing $y$ with $x$ results in an increasing tableau, then replace $y$ with $x$ and insert $y$ into the next row.
If replacing $y$ with $x$ does not result in an increasing tableau, 
then insert $y$ into the next row and do not change $R$.

\begin{exa} \label{Ex1} 
Let $Y$ be an increasing Young tableau given in Figure~\ref{fig:2}a. Suppose we wish to compute $(Y\xleftarrow{\,\mathrm{H}}2)$. Since the first row of $Y$ contains $2$, $4$ is inserted into the second row, whose largest value is 4. So, the algorithm terminates with $\alpha=0$, and $c=(3,2)$ is the corner in the third row and second column. The resulting tableau $Z$ is given in Figure~\ref{fig:2}b. On the other hand, suppose we wish to compute $(Y\xleftarrow{\,\mathrm{H}}5)$. In the first step, $5$ is inserted into the first row of $Y$, replacing $6$. The integer $6$ is then inserted into the second row, adding $6$ to the end of this row. The algorithm terminates with the tableau $Z$ in Figure~\ref{fig:2}c, $\alpha=1$, and $c=(2,3)$.
\end{exa}

\begin{figure}[h]
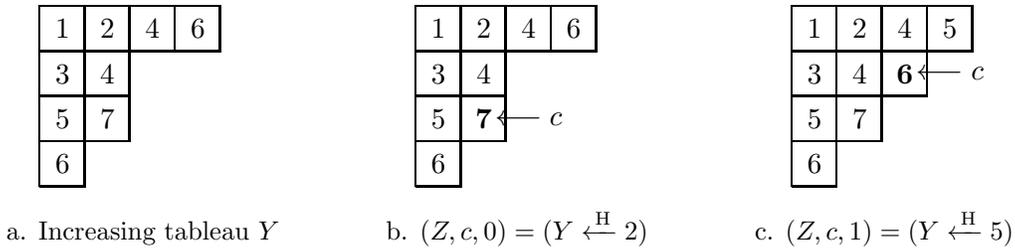

$$
\Einheit.2cm
\PfadDicke{.5pt}
\Pfad(0,0),222222222222\endPfad
\Pfad(3,0),222222222222\endPfad
\Pfad(6,3),222222222\endPfad
\Pfad(9,9),222\endPfad
\Pfad(12,9),222\endPfad
\Pfad(0,0),111\endPfad
\Pfad(0,3),111111\endPfad
\Pfad(0,6),111111\endPfad
\Pfad(0,9),111111111111\endPfad
\Pfad(0,12),111111111111\endPfad
\Label\ro{1}(1,10)
\Label\ro{3}(1,7)
\Label\ro{5}(1,4)
\Label\ro{6}(1,1)
\Label\ro{2}(4,10)
\Label\ro{4}(4,7)
\Label\ro{7}(4,4)
\Label\ro{4}(7,10)
\Label\ro{6}(10,10)
\hbox{\hskip5cm}
\Einheit.2cm
\PfadDicke{.5pt}
\Pfad(0,0),222222222222\endPfad
\Pfad(3,0),222222222222\endPfad
\Pfad(6,3),222222222\endPfad
\Pfad(9,9),222\endPfad
\Pfad(12,9),222\endPfad
\Pfad(0,0),111\endPfad
\Pfad(0,3),111111\endPfad
\Pfad(0,6),111111\endPfad
\Pfad(0,9),111111111111\endPfad
\Pfad(0,12),111111111111\endPfad
\Label\ro{1}(1,10)
\Label\ro{3}(1,7)
\Label\ro{5}(1,4)
\Label\ro{6}(1,1)
\Label\ro{2}(4,10)
\Label\ro{4}(4,7)
\Label\ro{\bf 7}(4,4)
\Label\ro{\longleftarrow c}(7,4)
\Label\ro{4}(7,10)
\Label\ro{6}(10,10)
\hbox{\hskip5cm}
\Einheit.2cm
\PfadDicke{.5pt}
\Pfad(0,0),222222222222\endPfad
\Pfad(3,0),222222222222\endPfad
\Pfad(6,3),222222222\endPfad
\Pfad(9,6),222222\endPfad
\Pfad(12,9),222\endPfad
\Pfad(0,0),111\endPfad
\Pfad(0,3),111111\endPfad
\Pfad(0,6),111111111\endPfad
\Pfad(0,9),111111111111\endPfad
\Pfad(0,12),111111111111\endPfad
\Label\ro{1}(1,10)
\Label\ro{3}(1,7)
\Label\ro{5}(1,4)
\Label\ro{6}(1,1)
\Label\ro{2}(4,10)
\Label\ro{4}(4,7)
\Label\ro{7}(4,4)
\Label\ro{4}(7,10)
\Label\ro{\bf 6}(7,7)
\Label\ro{\longleftarrow c}(10,7)
\Label\ro{5}(10,10)
\hbox{\hskip2.5cm}
$$
\centerline{\small a. Increasing  tableau $Y$
\hskip1.3cm
b. $(Z,c,0)= (Y\xleftarrow{\,\mathrm{H}}2)$
\hskip1.3cm
c. $(Z,c,1)= (Y\xleftarrow{\,\mathrm{H}}5)$}

\caption{Two examples of Hecke insertion.}
\label{fig:2}
\end{figure}

The Hecke insertion is reversible. Let $Z$ be an increasing Young tableau, $c$ be a corner of $Z$, 
and $\alpha\in\lbrace 0, 1\rbrace$. Reverse Hecke insertion applied to the triple $(Z, c, \alpha)$
produces a pair $(Y, x)$ of an increasing Young tableau $Y$ and a positive integer $x$ as follows. 
Let $y$ be the integer in the box $c$ of $Z$. If $\alpha=1$ then remove $y$. 
In any case, reverse insert $y$ into the row above the corner $c$.

Whenever $y$ is reverse inserted into a row $R$, let $x$ be the largest entry of $R$ such that $x<y$.
If replacing $x$ with $y$ results in an increasing tableau, then replace $x$ with $y$ and
$x$ is reverse inserted into the row above of $R$.
If replacing $x$ with $y$ does not result in an increasing tableau, 
then $x$ is reverse inserted into the row above of $R$ and $R$ remains the same.
This process is repeated until $R$ is the top row, then $x$ becomes the final output value, 
along with the modified tableau.

Hecke insertion is a generalization of the standard RSK algorithm. While the RSK algorithm associates a permutation to a pair of standard Young tableaux of the same shape, the Hecke insertion gives a correspondence between words and a pair $(P,\,Q)$, where $P$ is an increasing tableau and $Q$ is a set-valued tableau of the same shape. $P$ is the \emph{Hecke insertion tableau} while $Q$ is  the \emph{Hecke recording tableau}. {\it Set-valued tableaux}, introduced by Buch~\cite{Buch} in the study of the $K$-theory of Gra{\ss}mannians, are $\mathbb{N}$-fillings of a Young diagram assigning to each box a nonempty set of positive integers such that
the largest entry of a box is smaller than the smallest entry in the boxes directly to the right of it and directly below it. For a word of nonnegative integers $w$, the corresponding pair $(P,\,Q)$ is constructed as follows.

Let $w=w_1\cdots w_n$ be a word. The insertion tableau of $w$ is formed by 
recursively Hecke inserting the letters of $w$ from left to right:
$$P(w)=(\cdots((\emptyset\xleftarrow{\,\mathrm{H}} w_1)\xleftarrow{\,\mathrm{H}} w_2)\cdots \xleftarrow{\,\mathrm{H}} w_n).$$
The set-valued tableau of $w$ is obtained recursively as follows. 
Set $Q(\emptyset)=\emptyset$. At each step of the insertion of $w$,
let $Q(w_1 \cdots w_k)$ be obtained from $Q(w_1 \cdots w_{k-1})$ by labeling the special corner $c$,
in the insertion of $w_k$ into $P(w_1 \cdots w_{k-1})$ with the positive integer $k$.
Then $Q(w)=Q(w_1\cdots w_n)$ is the resulting set-valued tableau. Similarly as in the case of the RSK correspondence, the tableaux $Q$ determines the order in which numbers from $P$ need to be reverse inserted, so that the corresponding word can be recovered.

\begin{exa} \label{Ex2} 
Let $w=32412143$, $P(w)$ and $Q(w)$ are given in Figure~\ref{fig:3}.
\end{exa}

\begin{figure}[h]
$$
\Einheit.2cm
\PfadDicke{.5pt}
\Pfad(0,0),222222222\endPfad
\Pfad(3,0),222222222\endPfad
\Pfad(6,3),222222\endPfad
\Pfad(9,6),222\endPfad
\Pfad(0,0),111\endPfad
\Pfad(0,3),111111\endPfad
\Pfad(0,6),111111111\endPfad
\Pfad(0,9),111111111\endPfad
\Label\ro{1}(1,7)
\Label\ro{2}(1,4)
\Label\ro{3}(1,1)
\Label\ro{2}(4,7)
\Label\ro{4}(4,4)
\Label\ro{3}(7,7)
\Label\ro{P(w)=}(-4,4)
\hbox{\hskip5cm}
\Einheit.2cm
\PfadDicke{.5pt}
\Pfad(0,0),222222222\endPfad
\Pfad(3,0),222222222\endPfad
\Pfad(6,3),222222\endPfad
\Pfad(9,6),222\endPfad
\Pfad(0,0),111\endPfad
\Pfad(0,3),111111\endPfad
\Pfad(0,6),111111111\endPfad
\Pfad(0,9),111111111\endPfad
\Label\ro{1}(1,7)
\Label\ro{2}(1,4)
\Label\ro{4,6}(1,1)
\Label\ro{3}(4,7)
\Label\ro{5,8}(4,4)
\Label\ro{7}(7,7)
\Label\ro{Q(w)=}(-4,4)
\hskip1.8cm
$$
\caption{The pair $(P(w),Q(w))$ of an increasing and set-valued tableaux for $w=32412143$.}
\label{fig:3}
\end{figure}

In \cite{TY11}, Thomas and Yong prove that the insertion tableau of a word generated by the Hecke insertion
determines the lengths of the longest strictly increasing and strictly decreasing subsequences in the word.
That is to say, for a word $w$, let $\mathrm{lis}(w)$ (resp. $\mathrm{lds}(w)$) denote the length of the longest strictly increasing (resp. decreasing) subsequence of word $w$, and for an increasing tableau $P$, 
let $\mathrm{c}(P)$ (resp. $\mathrm{r}(P)$) denote the number of columns (resp. rows) of $P$, 
Thomas and Yong \cite{TY11} established the following theorem. 

\begin{thm} \label{Thm2} \cite{TY11}
Let $w$ be a word and let $P$ be the Hecke insertion tableau of $w$.
Then $\mathrm{lis}(w)=\mathrm{c}(P)$ and $\mathrm{lds}(w)=\mathrm{r}(P)$.
\end{thm}

\subsection{Hecke Growth Diagrams}

It is well-known that the pair of standard Young tableaux that correspond to a permutation $\pi$ via RSK can also be obtained using the \emph{growth diagrams} introduced by Fomin in \cite{Fomin86, Fomin94, Fomin95}. The growth diagrams for Hecke insertion was given by Patrias and Pylyavksyy in~\cite{PP}. In the following we follow \cite{PP} to give a description of the {\it Hecke growth diagram} of a word.

We represent a word $w=w_1w_2\cdots $ 
containing $a_1$ copies of $1$, $a_2$ copies of $2$, $\dots$,
$a_m$ copies of $m$ as a filling of 
an $m\times(a_1+a_2+\cdots+a_m)$ rectangle, 
by placing a $X$ into the $j$-th column and $w_j$-th row 
(we number  the rows  from bottom to top),
for $j=1,2,\ldots,a_1+a_2+\cdots+a_m$.
We say this is the {\it matrix representation} of $w$.
For example, the word $213312$ is represented by the filling of
Figure~\ref{fig:4}.

\begin{figure}[h]

$$
\Einheit.2cm
\PfadDicke{.5pt}
\Pfad(0,9),111111111111111111\endPfad
\Pfad(0,6),111111111111111111\endPfad
\Pfad(0,3),111111111111111111\endPfad
\Pfad(0,0),111111111111111111\endPfad
\Pfad(18,0),222222222\endPfad
\Pfad(15,0),222222222\endPfad
\Pfad(12,0),222222222\endPfad
\Pfad(9,0),222222222\endPfad
\Pfad(6,0),222222222\endPfad
\Pfad(3,0),222222222\endPfad
\Pfad(0,0),222222222\endPfad
\Label\ro{\text {\small$X$}}(1,4)
\Label\ro{\text {\small$X$}}(4,1)
\Label\ro{\text {\small$X$}}(7,7)
\Label\ro{\text {\small$X$}}(10,7)
\Label\ro{\text {\small$X$}}(13,1)
\Label\ro{\text {\small$X$}}(16,4)
\hskip4cm
$$
\caption{The matrix representation of $213312$.}
\label{fig:4}
\end{figure}

For a filling of a rectangle, we label each of the four corners of each square with a partition 
and possibly label the horizontal edge of the square with a positive integer $r$,
where $r$ records the row of the inner corner at which the Hecke insertion terminates.
We start by labeling all corners along the bottom row and left side of the diagram 
with the partition $\emptyset$ and then use the following {\it forward  local rules} 
to label the other corners of the squares and some of the horizontal edges of the squares.

Suppose we have the labels of the three corners of each square except the upper right one, for example,
given $\lambda$, $\mu$, $\nu$ in Figure~\ref{fig:7}, we can construct $\gamma$ and
possibly the positive integer for the horizontal edge of the square.

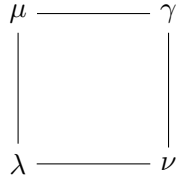
\begin{figure}[h]
$$
\begin{tikzpicture}[scale=1]
\node (A) at (-1,-1) {$\lambda$};
\node (B) at (1,-1) {$\nu$};
\node (C) at (-1,1) {$\mu$};
\node (D) at (1,1) {$\gamma$};
\draw (A) -- (B)
(C) -- (D)
(B) -- (D)
(A) -- (C);
\end{tikzpicture}
$$
\caption{A square of a growth diagram.}
\label{fig:7}
\end{figure}

Case 1. If the square contains an X:
\begin{itemize} 
\item[(F1)] If $\mu_1=\nu_1$, then $\gamma/\mu$ consists of one box in row 1.
\item[(F2)] If $\mu_1 \neq \nu_1$, then $\gamma=\mu$ and the edge between them is labeled by 
the row of the highest inner corner of $\mu$ (an inner corner is a maximally southeast box of $\mu$). 
\end{itemize}

Case 2. If the square does not contain an X and if either $\mu=\lambda$ or $\nu=\lambda$ with 
no label between $\lambda$ and $\nu$:
\begin{itemize}
\item[(F3)] If $\mu=\lambda$, then $\gamma = \nu$. If $\nu=\lambda$, then $\gamma = \mu$.  
\end{itemize}

Case 3. If the square does not contain an X, dose not satisfy the condition of Case 2, 
and if $\nu \nsubseteq \mu$:
\begin{itemize}
\item[(F4)] For Case 3, $\gamma = \nu \cup \mu$.  
\end{itemize}

Case 4. If the square does not contain an X, dose not satisfy the condition of Case 2, 
and if $\nu \subseteq \mu$:
\begin{itemize}
\item[(F5)] If $\nu/\lambda$ is one box in row $i$ and $\mu/\nu$ has no boxes in row $i+1$, 
then $\gamma/\mu$ is one box in row $i+1$.

\item[(F6)] If $\nu/\lambda$ is one box in row $i$ and $\mu/\nu$ has a box in row $i+1$, then $\gamma=\mu$ and the edge between them is labeled $i+1$.

\item[(F7)] If $\nu=\lambda$, the edge between them is labeled $i$, and there are no boxes of $\mu/\nu$
immediately to the right or immediately below the inner corner of $\nu$ in row $i$, then $\gamma=\mu$ with the edge between them labeled $i$.

\item[(F8)] If $\nu=\lambda$, the edge between them is labeled $i$, and there is a box of $\mu/\nu$ directly below
the inner corner of $\nu$ in row $i$, then $\gamma=\mu$ with the edge between them labeled $i+1$.

\item[(F9)] If $\nu=\lambda$, the edge between them is labeled $i$, and there is a box of $\mu/\nu$ immediately 
to the right of the inner corner of $\nu$ in row $i$ but no box of $\mu/\nu$ in row $i+1$, then $\gamma/\mu$ is one box in row $i+1$.

\item[(F10)] If $\nu=\lambda$, the edge between them is labeled $i$, and there is a box of $\mu/\nu$ immediately 
to the right of this inner corner of $\nu$ in row $i$ and a box of $\mu/\nu$ in row $i+1$, then $\gamma=\mu$ with the edge between them labeled $i+1$. 

\end{itemize}

We call the resulting diagram the {\it Hecke growth diagram}. 
For example, in Figure~\ref{fig:8} we have the Hecke growth diagram for the word $213312$.

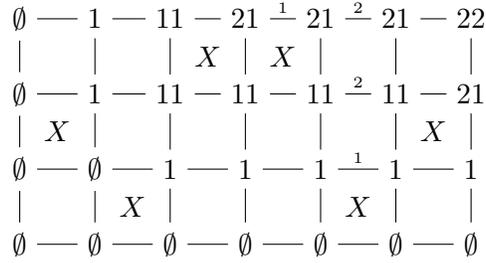
\begin{figure}
$$
\begin{tikzpicture}[scale=1]
\node (00) at (0,0) {$\emptyset$};
\node (10) at (1,0) {$\emptyset$};
\node (20) at (2,0) {$\emptyset$};
\node (30) at (3,0) {$\emptyset$};
\node (40) at (4,0) {$\emptyset$};
\node (50) at (5,0) {$\emptyset$};
\node (01) at (0,1) {$\emptyset$};
\node (.51.5) at (.5,1.5) {$X$};
\node (11) at (1,1) {$\emptyset$};
\node (21) at (2,1) {$1$};
\node (31) at (3,1) {$1$};
\node (41) at (4,1) {$1$};
\node (51) at (5,1) {$1$};
\node (02) at (0,2) {$\emptyset$};
\node (12) at (1,2) {$1$};
\node (22) at (2,2) {$11$};
\node (32) at (3,2) {$11$};
\node (1.50.5) at (1.5,0.5) {$X$};
\node (42) at (4,2) {$11$};
\node (52) at (5,2) {$11$};
\node (03) at (0,3) {$\emptyset$};
\node (13) at (1,3) {$1$};
\node (23) at (2,3) {$11$};
\node (33) at (3,3) {$21$};
\node (43) at (4,3) {$21$};
\node (53) at (5,3) {$21$};
\node (2.52.5) at (2.5,2.5) {$X$};
\node (3.52.5) at (3.5,2.5) {$X$};
\node (4.50.5) at (4.5,0.5) {$X$};
\node (5,51.5) at (5.5,1.5) {$X$};
\node (60) at (6,0) {$\emptyset$};
\node (61) at (6,1) {$1$};
\node (62) at (6,2) {$21$};
\node (63) at (6,3) {$22$};

\node (3.53.1) at (3.5,3.15) {\tiny 1};
\node (4.51.1) at (4.5,1.15) {\tiny 1};
\node (4.52.1) at (4.5,2.15) {\tiny 2};
\node (4.53.1) at (4.5,3.15) {\tiny 2};

\draw (00)--(10)--(20)--(30)--(40)--(50)--(60)
(01)--(11)--(21)--(31)--(41)--(51)--(61)
(02)--(12)--(22)--(32)--(42)--(52)--(62)
(03)--(13)--(23)--(33)--(43)--(53)--(63)

(00)--(01)--(02)--(03)
(10)--(11)--(12)--(13)
(20)--(21)--(22)--(23)
(30)--(31)--(32)--(33)
(40)--(41)--(42)--(43)
(50)--(51)--(52)--(53)
(60)--(61)--(62)--(63);
\end{tikzpicture}
$$
\caption{The Hecke growth diagram of the word $213312$.}
\label{fig:8}
\end{figure}

In the Hecke growth diagram of a word $w$, the sequence of partitions 
$\emptyset=\mu_0\subseteq\mu_1\subseteq\ldots\subseteq\mu_k$ along the upper border of 
the Hecke growth diagram corresponds to a set-valued tableau $Q(w)$:
if $\mu_{i-1}\neq\mu_i$, then we put $i$ into the square by which $\mu_{i-1}$ and $\mu_i$ differ.
Otherwise, $\mu_{i-1}=\mu_i$ and the edge between $\mu_{i-1}$ and $\mu_i$ is labeled $c$,
then we put $i$ into the box at the end of row $c$ of $\mu_{i-1}$.
We also have the sequence of partitions 
$\emptyset=\nu_0\subseteq\nu_1\subseteq\ldots\subseteq\nu_n$
along the right border of the Hecke growth diagram corresponds to 
an increasing tableau $P(w)$:
we put $i$'s into the squares of $\nu_i/\nu_{i-1}$.
For example, from the chains of partitions on the rightmost edge and uppermost edge of the diagram in Figure~\ref{fig:8},
we obtain the insertion and recording tableau in Figure~\ref{fig:PQ}.

\begin{figure}[h]
$$
\Einheit.2cm
\PfadDicke{.5pt}
\Pfad(0,0),222222\endPfad
\Pfad(3,0),222222\endPfad
\Pfad(6,0),222222\endPfad
\Pfad(0,0),111111\endPfad
\Pfad(0,3),111111\endPfad
\Pfad(0,6),111111\endPfad
\Label\ro{1}(1,4)
\Label\ro{2}(4,4)
\Label\ro{2}(1,1)
\Label\ro{3}(4,1)
\Label\ro{P(213312)=}(-6,2.5)
\hbox{\hskip5cm}
\Einheit.2cm
\PfadDicke{.5pt}
\Pfad(0,0),222222\endPfad
\Pfad(3,0),222222\endPfad
\Pfad(6,0),222222\endPfad
\Pfad(0,0),111111\endPfad
\Pfad(0,3),111111\endPfad
\Pfad(0,6),111111\endPfad
\Label\ro{1}(1,4)
\Label\ro{3,4}(4,4)
\Label\ro{2,5}(1,1)
\Label\ro{6}(4,1)
\Label\ro{Q(213312)=}(-6,2.5)
\hskip1.8cm
$$
\caption{Hecke insertion and recording tableaux obtained from the growth diagram.}
\label{fig:PQ}
\end{figure}

In \cite{PP}, Patrias and Pylyavskyy also formulate the {\it backward local rules},
that is, given $\gamma$, $\mu$, $\nu$ and the edge label between $\mu$ and $\gamma$,
one can reconstruct $\lambda$, the edge label between $\lambda$ and $\nu$,
and the filling of the square.

\begin{thm} \label{Thm4} \cite[Theorem 4.16 ]{PP}
For any word $w$, the increasing tableau $P(w)$ and set-valued tableau $Q(w)$
obtained from the sequence of partitions along the right border of the Hecke growth diagram for $w$
and along the upper border of the Hecke growth diagram for $w$, respectively, are the Hecke insertion tableau and the Hecke recording tableau for $w$.
\end{thm}

\subsection{Jeu de Taquin for Increasing Tableaux} 

In this section, we give the description of jeu de taquin for increasing tableaux which was introduced
by Thomas and Yong in \cite{TY09}.

An {\it increasing tableau} $T$ of skew shape $\lambda/\mu$ is 
a $\mathbb{N}$-filling of $\lambda/\mu$ assigning to each box a positive integer
such that the entries of $T$ strictly increasing down columns and from left to right along rows.
Let $\mathrm{INC}(\lambda/\mu)$ be the set of these increasing tableaux.

For $T\in \mathrm{INC}(\lambda/\mu)$, an {\it inner corner} is a maximally southeast box of $\mu$.
Let $C$ be a set of some inner corners of $\mu$, and mark the boxes in $C$ with $\bullet$.
The {\it jeu de taquin} for $T$ and $C$ can be described as follows.
We begin with the shape $T_1$ which is made of boxes with entries $\bullet$ and $1$ in $T\cup C$.
If $1$ is directly below or right of $\bullet$, then we swap $\bullet$ and $1$ in $T_1$,
otherwise we keep the same labels for these boxes in $T_1$. 
Next we consider the shape $T_2$ which is made of boxes with entries $\bullet$ and $2$.
If $2$ is directly below or right of $\bullet$, then we swap $\bullet$ and $2$ in $T_2$,
otherwise we keep the same labels for these boxes in $T_2$. 
We repeat the above processes until the $\bullet$'s become the inner corners of $\lambda$.
The final placement of the numerical entries gives the jeu de taquin for $T$ and $C$ and we write $jdt_C(T)$.
It is easy to see that $jdt_C(T)$ is an increasing tableau as well.

\begin{exa} \label{Ex3} 
Let $T$ be an increasing tableau as given in Figure~\ref{fig:5} and $C$ indicated with $\bullet$. 
See Figure~\ref{fig:5}e for $jdt_C(T)$.
\end{exa}

\begin{figure}[h]
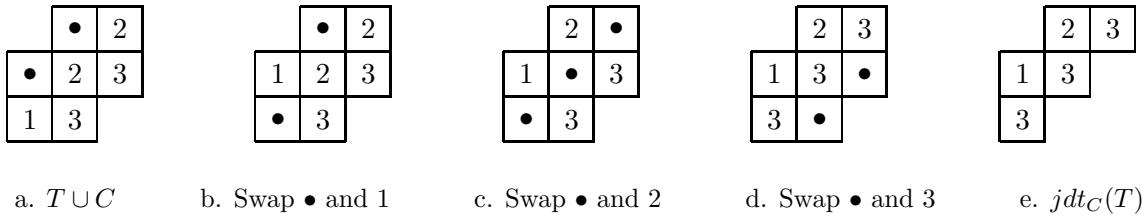

$$
\Einheit.2cm
\PfadDicke{.5pt}
\Pfad(0,0),222222\endPfad
\Pfad(3,0),222222222\endPfad
\Pfad(6,0),222222222\endPfad
\Pfad(9,3),222222\endPfad
\Pfad(0,0),111111\endPfad
\Pfad(0,3),111111111\endPfad
\Pfad(0,6),111111111\endPfad
\Pfad(3,9),111111\endPfad
\Label\ro{\bullet}(1,4)
\Label\ro{1}(1,1)
\Label\ro{\bullet}(4,7)
\Label\ro{2}(4,4)
\Label\ro{3}(4,1)
\Label\ro{2}(7,7)
\Label\ro{3}(7,4)
\hbox{\hskip3.3cm}
\Einheit.2cm
\PfadDicke{.5pt}
\Pfad(0,0),222222\endPfad
\Pfad(3,0),222222222\endPfad
\Pfad(6,0),222222222\endPfad
\Pfad(9,3),222222\endPfad
\Pfad(0,0),111111\endPfad
\Pfad(0,3),111111111\endPfad
\Pfad(0,6),111111111\endPfad
\Pfad(3,9),111111\endPfad
\Label\ro{1}(1,4)
\Label\ro{\bullet}(1,1)
\Label\ro{\bullet}(4,7)
\Label\ro{2}(4,4)
\Label\ro{3}(4,1)
\Label\ro{2}(7,7)
\Label\ro{3}(7,4)
\hbox{\hskip3.3cm}
\Einheit.2cm
\PfadDicke{.5pt}
\Pfad(0,0),222222\endPfad
\Pfad(3,0),222222222\endPfad
\Pfad(6,0),222222222\endPfad
\Pfad(9,3),222222\endPfad
\Pfad(0,0),111111\endPfad
\Pfad(0,3),111111111\endPfad
\Pfad(0,6),111111111\endPfad
\Pfad(3,9),111111\endPfad
\Label\ro{1}(1,4)
\Label\ro{\bullet}(1,1)
\Label\ro{2}(4,7)
\Label\ro{\bullet}(4,4)
\Label\ro{3}(4,1)
\Label\ro{\bullet}(7,7)
\Label\ro{3}(7,4)
\hbox{\hskip3.3cm}
\Einheit.2cm
\PfadDicke{.5pt}
\Pfad(0,0),222222\endPfad
\Pfad(3,0),222222222\endPfad
\Pfad(6,0),222222222\endPfad
\Pfad(9,3),222222\endPfad
\Pfad(0,0),111111\endPfad
\Pfad(0,3),111111111\endPfad
\Pfad(0,6),111111111\endPfad
\Pfad(3,9),111111\endPfad
\Label\ro{1}(1,4)
\Label\ro{3}(1,1)
\Label\ro{2}(4,7)
\Label\ro{3}(4,4)
\Label\ro{\bullet}(4,1)
\Label\ro{3}(7,7)
\Label\ro{\bullet}(7,4)
\hbox{\hskip3.3cm}
\Einheit.2cm
\PfadDicke{.5pt}
\Pfad(0,0),222222\endPfad
\Pfad(3,0),222222222\endPfad
\Pfad(6,3),222222\endPfad
\Pfad(9,6),222\endPfad
\Pfad(0,0),111\endPfad
\Pfad(0,3),111111\endPfad
\Pfad(0,6),111111111\endPfad
\Pfad(3,9),111111\endPfad
\Label\ro{1}(1,4)
\Label\ro{3}(1,1)
\Label\ro{2}(4,7)
\Label\ro{3}(4,4)
\Label\ro{3}(7,7)
\hskip2cm
$$
\centerline{\small a. $T\cup C$
\hskip1cm
b. Swap $\bullet$ and $1$
\hskip1cm
c. Swap $\bullet$ and $2$
\hskip1cm
d. Swap $\bullet$ and $3$
\hskip1cm
e. $jdt_C(T)$}
\caption{An example of jeu de taquin for increasing tableaux.}
\label{fig:5}
\end{figure}

 \emph{Reverse jeu de taquin} for $T^\prime\in \mathrm{INC}(\lambda/\mu)$,
where $\bullet$'s are located in \emph{outer corners} - the maximally northwest boxes in $\gamma/\lambda$ - is performed in a similar manner.
Let $C^\prime$ be a set of some outer corners, and mark the boxes in $C^\prime$ with $\bullet$.
We begin with the shape $T_n^\prime$ which is made of boxes with entries $\bullet$ and $n$ (the maximum value of $T^\prime$)
in $T^\prime\cup C^\prime$.
If $n$ is directly above or left of $\bullet$, then we swap $\bullet$ and $n$ in $T_n^\prime$,
otherwise we keep the same labels for these boxes in $T_n^\prime$. 
Next we consider the shape $T_{n-1}^\prime$ which is made of boxes with entries $\bullet$ and $n-1$.
If $n-1$ is directly above or left of $\bullet$, then we swap $\bullet$ and $n-1$ in $T_{n-1}^\prime$,
otherwise we keep the same labels for these boxes in $T_{n-1}^\prime$. 
We can repeat the above processes until the $\bullet$'s become the outer corners of $\lambda/\mu$.
The final placement of the numerical entries gives the reverse jeu de taquin for $T^\prime$ and $C^\prime$ and write $revjdt_{C^\prime}(T^\prime)$. See Figure~\ref{fig:6} for an example.

\begin{figure}[h]
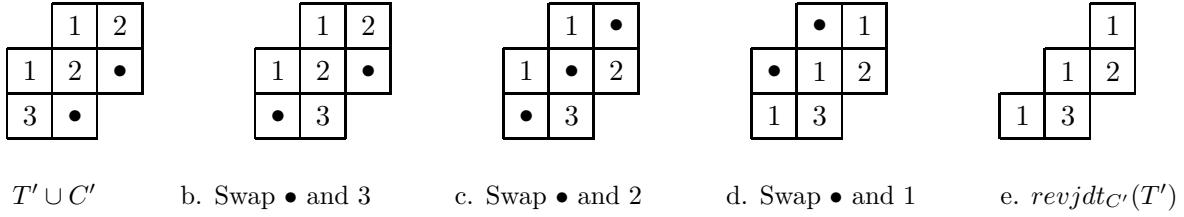

$$
\Einheit.2cm
\PfadDicke{.5pt}
\Pfad(0,0),222222\endPfad
\Pfad(3,0),222222222\endPfad
\Pfad(6,0),222222222\endPfad
\Pfad(9,3),222222\endPfad
\Pfad(0,0),111111\endPfad
\Pfad(0,3),111111111\endPfad
\Pfad(0,6),111111111\endPfad
\Pfad(3,9),111111\endPfad
\Label\ro{1}(1,4)
\Label\ro{3}(1,1)
\Label\ro{1}(4,7)
\Label\ro{2}(4,4)
\Label\ro{\bullet}(4,1)
\Label\ro{2}(7,7)
\Label\ro{\bullet}(7,4)
\hbox{\hskip3.3cm}
\Einheit.2cm
\PfadDicke{.5pt}
\Pfad(0,0),222222\endPfad
\Pfad(3,0),222222222\endPfad
\Pfad(6,0),222222222\endPfad
\Pfad(9,3),222222\endPfad
\Pfad(0,0),111111\endPfad
\Pfad(0,3),111111111\endPfad
\Pfad(0,6),111111111\endPfad
\Pfad(3,9),111111\endPfad
\Label\ro{1}(1,4)
\Label\ro{\bullet}(1,1)
\Label\ro{1}(4,7)
\Label\ro{2}(4,4)
\Label\ro{3}(4,1)
\Label\ro{2}(7,7)
\Label\ro{\bullet}(7,4)
\hbox{\hskip3.3cm}
\Einheit.2cm
\PfadDicke{.5pt}
\Pfad(0,0),222222\endPfad
\Pfad(3,0),222222222\endPfad
\Pfad(6,0),222222222\endPfad
\Pfad(9,3),222222\endPfad
\Pfad(0,0),111111\endPfad
\Pfad(0,3),111111111\endPfad
\Pfad(0,6),111111111\endPfad
\Pfad(3,9),111111\endPfad
\Label\ro{1}(1,4)
\Label\ro{\bullet}(1,1)
\Label\ro{1}(4,7)
\Label\ro{\bullet}(4,4)
\Label\ro{3}(4,1)
\Label\ro{\bullet}(7,7)
\Label\ro{2}(7,4)
\hbox{\hskip3.3cm}
\Einheit.2cm
\PfadDicke{.5pt}
\Pfad(0,0),222222\endPfad
\Pfad(3,0),222222222\endPfad
\Pfad(6,0),222222222\endPfad
\Pfad(9,3),222222\endPfad
\Pfad(0,0),111111\endPfad
\Pfad(0,3),111111111\endPfad
\Pfad(0,6),111111111\endPfad
\Pfad(3,9),111111\endPfad
\Label\ro{\bullet}(1,4)
\Label\ro{1}(1,1)
\Label\ro{\bullet}(4,7)
\Label\ro{1}(4,4)
\Label\ro{3}(4,1)
\Label\ro{1}(7,7)
\Label\ro{2}(7,4)
\hbox{\hskip3.3cm}
\Einheit.2cm
\PfadDicke{.5pt}
\Pfad(0,0),222\endPfad
\Pfad(3,0),222222\endPfad
\Pfad(6,0),222222222\endPfad
\Pfad(9,3),222222\endPfad
\Pfad(0,0),111111\endPfad
\Pfad(0,3),111111111\endPfad
\Pfad(3,6),111111\endPfad
\Pfad(6,9),111\endPfad
\Label\ro{1}(1,1)
\Label\ro{1}(4,4)
\Label\ro{3}(4,1)
\Label\ro{1}(7,7)
\Label\ro{2}(7,4)
\hskip2cm
$$
\centerline{\small a. $T^\prime\cup C^\prime$
\hskip1cm
b. Swap $\bullet$ and $3$
\hskip1cm
c. Swap $\bullet$ and $2$
\hskip1cm
d. Swap $\bullet$ and $1$
\hskip1cm
e. $revjdt_{C^\prime}(T^\prime)$}
\caption{An example of reverse jeu de taquin for increasing tableaux.}
\label{fig:6}
\end{figure}

Two increasing skew tableaux $T$ and $T^\prime$ are
{\it $K$-jeu de taquin equivalent} if $T$ can be obtained by applying a sequence of 
jeu de taquin or reverse jeu de taquin operations to $T^\prime$~\cite{BuchSamuel}. 

\subsection{$K$-Knuth Equivalence on Words and Increasing Tableaux}

Two permutations have the same insertion tableau by the standard RSK algorithm if and only
if they are {\it Knuth equivalent}, i.e., one can be obtained from the other via a finite series of 
applications of the Knuth relations: (1) for $x<y<z$, $xzy\sim zxy$; (2) for $x<y<z$, $yxz\sim yzx$.
Knuth equivalence on permutations has a generalization referred to as $K$-Knuth equivalence on words,
which was defined by Buch and Samuel in ~\cite{BuchSamuel} and motivated by~\cite{TY09}.

Two words are said to be {\it $K$-Knuth equivalent} if one can be obtained from the other via a finite series
of applications of the following $K$-Knuth relations:
\begin{align*}
xzy &\equiv zxy,\qquad (x < y < z) \\
yxz &\equiv yzx,\qquad (x < y < z) \\
x &\equiv xx, \\
xyx &\equiv yxy.
\end{align*}

We say that two words are  \emph{Hecke insertion equivalent} if they have the same Hecke insertion tableau. Hecke insertion equivalence implies $K$-Knuth equivalence~\cite{BuchSamuel}. However, unlike the RSK case, the converse is not true - $K$-Knuth equivalent words may have different Hecke insertion tableaux.

Let $T$ be an increasing tableau, we write $\mathrm{row}(T)$ for the reading word of $T$,
which is obtained by reading the entries of $T$ from left to right along each row,
starting from the bottom row and moving upward. 
For example, the reading word of $P(w)$ in Figure~\ref{fig:3} is $324123$.
For two increasing tableaux $T$ and $T^\prime$, if $\mathrm{row}(T)\equiv\mathrm{row}(T^\prime)$,
we say $T$ is {\it $K$-Knuth equivalent} to $T^\prime$, and write $T\equiv T^\prime$.
We list some results which will be used for the proof of 
the main theorem in Section~\ref{S:proof}.

\begin{lem} \label{Pro1} \cite[Lemma 5.5]{BuchSamuel}
Let $[a,b]$ be an integer interval. Let $w_1$ and $w_2$ be $K$-Knuth equivalent words. 
For $i=1,2$, let $w_i|_{[a,b]}$ be the word obtained from $w_i$ by deleting all integers not 
contained in the interval $[a,b]$. Then $w_1|_{[a,b]}$ and $w_2|_{[a,b]}$ are
$K$-Knuth equivalent words.
\end{lem}

\begin{thm} \label{Thm3} \cite[Theorem 6.2]{BuchSamuel}
Let $T$ and $T^\prime$ be increasing tableaux. Then $T$ and $T^\prime$ are $K$-Knuth equivalent
if and only if $T$ and $T^\prime$ are $K$-jeu de taquin equivalent.
\end{thm}

\begin{prop} \label{Pro2} \cite[Proposition 37]{GMPPRST}
If $w_1\equiv w_2$ then $\mathrm{lis}(w_1) = \mathrm{lis}(w_2)$ and $\mathrm{lds}(w_1) = \mathrm{lds}(w_2)$.
\end{prop}

\begin{lem} \label{Pro3} \cite[Lemma 58]{GMPPRST}
If $w$ be a word and $P(w)$ be the Hecke insertion tableau of $w$, 
then $w$ and $\mathrm{row}(P(w))$ are $K$-Knuth equivalent.
\end{lem}

\begin{prop}\label{Pro4} \cite[Proposition 3.2]{CGP}
Let $w=w_1w_2\cdots w_n$ be a word of positive integers, and $k$ be the maximal element appearing in $w$.
Let $w'=a_1a_2\cdots a_m$ be the word obtained from $w$ by deleting the elements equal to $k$.
Assume that $T$ is the insertion tableau of $w$ and $T'$ is the insertion tableau of $w'$.
Then  $T'$ is obtained from $T$ by deleting the squares occupied with $k$.
\end{prop}

\section{Proof of the Main Theorem}\label{S:proof}


Recall that $N(\mathcal{M};n;ne=u,se=v)$ denotes the number of 01-fillings 
of the polyomino $\mathcal{M}$  such that the sum of entries equal to $n$, the longest $ne$-chain has length $u$, and 
the longest $se$-chain has length $v$. The goal of this section is to provide proof of Theorem~\ref{Thm1} which relates the fillings of the stack polyominos $\mathcal{M}$ and $\sigma \mathcal{M}$.

Let $\mathcal{M}$ be a stack polyomino. Since the rows are left-justified, it is determined by its sequence of row lengths  $(r_1, r_2, \ldots, r_k)$ counted from bottom up. For example, the stack polyomino in Figure~\ref{fig:1}b  is represented by $(2,3,5,6,6,4,1)$. If the bottom row of $\mathcal{M}$ is not the longest then it can be moved up and inserted in a new position above the longest row so that the resulting shape is still a stack polyomino. Such a position always exists and if there are multiple positions where the bottom row can be reinserted we think of inserting it in the highest position possible. This way, by applying only such ``move the bottom row up'' moves the rows of $\mathcal{M}$ can be reordered so that the resulting polyomino is a Ferrers shape $\mathcal{F}$ in French notation (the row lengths weakly increase top to bottom). For example, $\mathcal{M}=(3,4,5,4,2)\rightarrow(4,5,4,3,2)\rightarrow(5,4,4,3,2)=\mathcal{F}$. The Ferrers shape $\mathcal{F}$ depends only on the set of row lengths of $\mathcal{M}$ and not the actual sequence, so any stack polyomino $\sigma \mathcal{M}$ can be transformed to the same Ferrers shape $\mathcal{F}$ by applying ``move the bottom row up'' moves. Theoerefore, to prove Theorem~\ref{Thm1}, it suffices to prove only 

\begin{prop} \label{prop:moveup}
Let $\mathcal{M}$ be a stack polyomino which is not a Ferrers shape and let $\mathcal{M}'$ be obtained by moving the bottom row of $\mathcal{M}$ up as much as possible so that 
the resulting polyomino is also a stack polyomino. Then \[N(\mathcal{M};n;ne=u,se=v) = N(\mathcal{M}';n;ne=u,se=v).\] 
\end{prop}

\begin{proof}
Let $\mathcal{F}(\mathcal{M})$ denote the set of fillings of a polyomino $\mathcal{M}$. We prove this claim by constructing a bijection $f : \mathcal{F}(\mathcal{M}) \rightarrow \mathcal{F}(\mathcal{M}')$ between the fillings of $\mathcal{M}$ and $\mathcal{M}'$ which preserves the lengths of the longest \emph{ne}- and \emph{se}-chains. For $M \in \mathcal{F}(\mathcal{M})$, the filling $f(M)$ of $\mathcal{M}'$ is defined by modifying the filling inside one maximal rectangle as follows.

Let $r$ be the bottom row of $\mathcal{M}$ and let $R$ be the largest rectangle contained in $\mathcal{M}$ that contains row $r$. Let $R'$ be the maximal rectangle in $\mathcal{M}'$ of the same size as $R$. We first define a bijection $\varphi_{R ,R ^\prime}: \mathcal{F}(R)\rightarrow\mathcal{F}(R^\prime)$ between the fillings of $R$ and $R'$ as follows. Let $T\in \mathcal{F}(R)$. In general, $T$ may contain some empty rows and columns. If the bottom row of $T$ is empty then $\varphi_{R ,R'}(T)$ is the filling $T'$ whose top row is empty while the $k$-th row, $k \geq 1$ is the same as the $(k+1)$-st row of $T$. Otherwise, if $1 < a_1 < a_2 < \ldots < a_i$ are the empty rows of $T$ then the rows $a_1-1, a_2-1, \ldots, a_i-1$ in $T' = \varphi_{R,R^\prime}(T)$ are empty. If $b_1, b_2, \ldots, b_j$ are the empty columns of $T$ then $b_1, b_2, \ldots, b_j$ are also empty columns in $T'$. Let $w$ be the word that corresponds to the filling $T$ when the empty rows and columns are ignored. For example, the word that corresponds to the rectangle $R$ in Figure~\ref{fig:11} is  $w = 236126415$.

\begin{figure}[h]
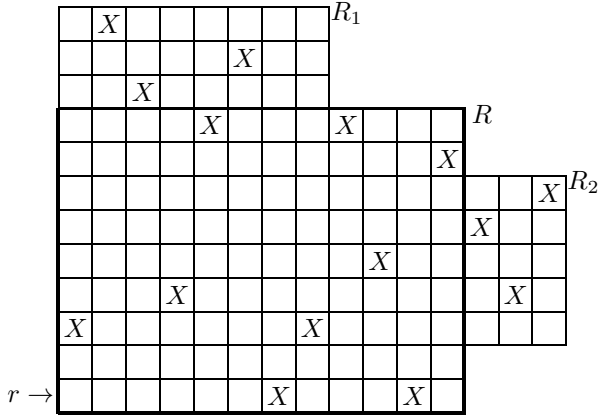

\begin{center}
$$
\Einheit.15cm
\PfadDicke{.3pt}
\Pfad(45,6),222222222222222\endPfad
\Pfad(42,6),222222222222222\endPfad
\Pfad(39,6),222222222222222\endPfad
\Pfad(33,0),222222222222222222222222222\endPfad
\Pfad(30,0),222222222222222222222222222\endPfad
\Pfad(27,0),222222222222222222222222222\endPfad
\Pfad(24,0),222222222222222222222222222222222222\endPfad
\Pfad(21,0),222222222222222222222222222222222222\endPfad
\Pfad(18,0),222222222222222222222222222222222222\endPfad
\Pfad(15,0),222222222222222222222222222222222222\endPfad
\Pfad(12,0),222222222222222222222222222222222222\endPfad
\Pfad(9,0),222222222222222222222222222222222222\endPfad
\Pfad(6,0),222222222222222222222222222222222222\endPfad
\Pfad(3,0),222222222222222222222222222222222222\endPfad
\Pfad(0,0),222222222222222222222222222222222222\endPfad
\Pfad(0,36),111111111111111111111111\endPfad
\Pfad(0,33),111111111111111111111111\endPfad
\Pfad(0,30),111111111111111111111111\endPfad
\Pfad(0,24),111111111111111111111111111111111111\endPfad
\Pfad(0,21),111111111111111111111111111111111111111111111\endPfad
\Pfad(0,18),111111111111111111111111111111111111111111111\endPfad
\Pfad(0,15),111111111111111111111111111111111111111111111\endPfad
\Pfad(0,12),111111111111111111111111111111111111111111111\endPfad
\Pfad(0,9),111111111111111111111111111111111111111111111\endPfad
\Pfad(0,6),111111111111111111111111111111111111111111111\endPfad
\Pfad(0,3),111111111111111111111111111111111111\endPfad
\PfadDicke{1pt}
\Pfad(0,27),111111111111111111111111111111111111\endPfad
\Pfad(0,0),111111111111111111111111111111111111\endPfad
\Pfad(0,0),222222222222222222222222222\endPfad
\Pfad(36,0),222222222222222222222222222\endPfad
\Label\ro{\text {\small$X$}}(1,7)
\Label\ro{\text {\small$X$}}(4,34)
\Label\ro{\text {\small$X$}}(7,28)
\Label\ro{\text {\small$X$}}(10,10)
\Label\ro{\text {\small$X$}}(13,25)
\Label\ro{\text {\small$X$}}(16,31)
\Label\ro{\text {\small$X$}}(19,1)
\Label\ro{\text {\small$X$}}(22,7)
\Label\ro{\text {\small$X$}}(25,25)
\Label\ro{\text {\small$X$}}(28,13)
\Label\ro{\text {\small$X$}}(31,1)
\Label\ro{\text {\small$X$}}(34,22)
\Label\ro{\text {\small$X$}}(37,16)
\Label\ro{\text {\small$X$}}(40,10)
\Label\ro{\text {\small$X$}}(43,19)
\Label\ro{\text {\small$R_1$}}(25,35)
\Label\ro{\text {\small$R$}}(37,26)
\Label\ro{\text {\small$R_2$}}(46,20)
\Label\ro{\text {\small$r\rightarrow$}}(-3,1)
\hskip5cm
$$
\end{center}
\caption{A filling of the stack polyomino $\mathcal{M}$.}
\label{fig:11}
\end{figure}

Let $(P, Q)$ be the Hecke insertion tableau and the Hecke recording tableau for $w$, respectively. Let $P'$ be the tableau obtained in the following way:

Step 1. Let $C$ be the top left box of $P$. Replace the entry $1$ appearing in $C$ with a $\bullet$. 

Step 2. Swap adjacent $\bullet$'s and $2$'s, then swap adjacent $\bullet$'s and $3$'s, etc, until the $\bullet$'s 
have been swapped with all the entries of $P$. The increasing tableau obtained this way is $jdt_{\{1\}}(P/1)$.

Step 3. Decrease all entries by 1 and then replace the $\bullet$'s with the maximum value of $P$. The resulting increasing tableau is $P^\prime$.

See Figure~\ref{fig:13} for an illustration of obtaining $P'$  for the Hecke insertion tableau $P$ of the word $w = 236126415$.

\begin{figure}[b!]
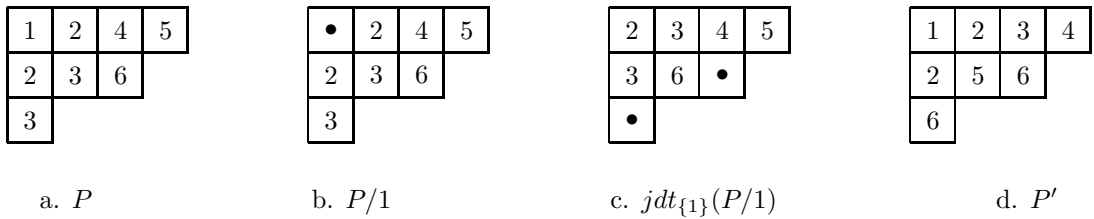

$$
\Einheit.2cm
\PfadDicke{.5pt}
\Pfad(12,6),222\endPfad
\Pfad(9,3),222222\endPfad
\Pfad(6,3),222222\endPfad
\Pfad(3,0),222222222\endPfad
\Pfad(0,0),222222222\endPfad
\Pfad(0,9),111111111111\endPfad
\Pfad(0,6),111111111111\endPfad
\Pfad(0,3),111111111\endPfad
\Pfad(0,0),111\endPfad
\Label\ro{\text {\small$1$}}(1,7)
\Label\ro{\text {\small$2$}}(1,4)
\Label\ro{\text {\small$3$}}(1,1)
\Label\ro{\text {\small$2$}}(4,7)
\Label\ro{\text {\small$3$}}(4,4)
\Label\ro{\text {\small$4$}}(7,7)
\Label\ro{\text {\small$6$}}(7,4)
\Label\ro{\text {\small$5$}}(10,7)
\hbox{\hskip4cm}
\Einheit.2cm
\PfadDicke{.5pt}
\Pfad(12,6),222\endPfad
\Pfad(9,3),222222\endPfad
\Pfad(6,3),222222\endPfad
\Pfad(3,0),222222222\endPfad
\Pfad(0,0),222222222\endPfad
\Pfad(0,9),111111111111\endPfad
\Pfad(0,6),111111111111\endPfad
\Pfad(0,3),111111111\endPfad
\Pfad(0,0),111\endPfad
\Label\ro{\text {\small$\bullet$}}(1,7)
\Label\ro{\text {\small$2$}}(1,4)
\Label\ro{\text {\small$3$}}(1,1)
\Label\ro{\text {\small$2$}}(4,7)
\Label\ro{\text {\small$3$}}(4,4)
\Label\ro{\text {\small$4$}}(7,7)
\Label\ro{\text {\small$6$}}(7,4)
\Label\ro{\text {\small$5$}}(10,7)
\hbox{\hskip4cm}
\Einheit.2cm
\PfadDicke{.5pt}
\Pfad(12,6),222\endPfad
\Pfad(9,3),222222\endPfad
\Pfad(6,3),222222\endPfad
\Pfad(3,0),222222222\endPfad
\Pfad(0,0),222222222\endPfad
\Pfad(0,9),111111111111\endPfad
\Pfad(0,6),111111111111\endPfad
\Pfad(0,3),111111111\endPfad
\Pfad(0,0),111\endPfad
\Label\ro{\text {\small$2$}}(1,7)
\Label\ro{\text {\small$3$}}(1,4)
\Label\ro{\text {\small$\bullet$}}(1,1)
\Label\ro{\text {\small$3$}}(4,7)
\Label\ro{\text {\small$6$}}(4,4)
\Label\ro{\text {\small$4$}}(7,7)
\Label\ro{\text {\small$\bullet$}}(7,4)
\Label\ro{\text {\small$5$}}(10,7)
\hbox{\hskip4cm}
\Einheit.2cm
\PfadDicke{.5pt}
\Pfad(12,6),222\endPfad
\Pfad(9,3),222222\endPfad
\Pfad(6,3),222222\endPfad
\Pfad(3,0),222222222\endPfad
\Pfad(0,0),222222222\endPfad
\Pfad(0,9),111111111111\endPfad
\Pfad(0,6),111111111111\endPfad
\Pfad(0,3),111111111\endPfad
\Pfad(0,0),111\endPfad
\Label\ro{\text {\small$1$}}(1,7)
\Label\ro{\text {\small$2$}}(1,4)
\Label\ro{\text {\small$6$}}(1,1)
\Label\ro{\text {\small$2$}}(4,7)
\Label\ro{\text {\small$5$}}(4,4)
\Label\ro{\text {\small$3$}}(7,7)
\Label\ro{\text {\small$6$}}(7,4)
\Label\ro{\text {\small$4$}}(10,7)
\hskip3cm
$$
\centerline{\small 
\hskip2cm
a. $P$
\hskip2.8cm
b. $P/1$
\hskip2.8cm
c. $jdt_{\{1\}}(P/1)$
\hskip2.8cm
d. $P^\prime$
\hskip2.5cm}
\caption{An example of constructing  $P'$ for the tableau $P$.}
\label{fig:13}
\end{figure}

Let $w'$ be the word that corresponds to  $(P^\prime, Q)$, obtained  by reverse Hecke insertion. Then $T^\prime$ is the filling of $R'$ which, when the empty rows $a_1-1, a_2-1, \ldots, a_i-1$ and the empty columns $b_1, b_2, \ldots, b_j$ are deleted, is the matrix representation of $w^\prime$. Finally, we set  $\varphi_{R,R^\prime}(T)=T^\prime$ for $T\in \mathcal{F}(R)$.

The map $\varphi_{R,R^\prime}$ is clearly invertible since the tableau $P$ can be constructed from $P'$ by doing the following:

Step 1. Replace the maximal entries $m$ appearing in $P^\prime$ with $\bullet$.

Step 2. Swap adjacent $\bullet$'s and $m-1$'s, then swap $\bullet$'s and $m-2$'s, etc, until the $\bullet$'s  have been swapped with all the entries of $P^\prime$. 

Step 3. Increase all entries by 1 and then replace the $\bullet$ with $1$. The resulting increasing tableau is $P$.

We now define $f: \mathcal{F}(\mathcal{M})\longrightarrow\mathcal{F}(\mathcal{M}')$ as follows. Let $T \in \mathcal{F}(\mathcal{M})$. Then $f(T)$ is the filling of $\mathcal{M}'$ in which all the entries outside of the rectangle $R'$ in $\mathcal{M}'$ are the same as in the filling of $\mathcal{M}$, while the filling of the rectangle $R'$ is $\varphi_{R, R'} (T |_{R})$ obtained by applying $\varphi_{R, R'}$ to the filling $T$ restricted on $R$. Figure~\ref{fig:12} shows the result of applying $f$ to  the filling of $\mathcal{M}$ in Figure~\ref{fig:11}.

\begin{figure}[h]
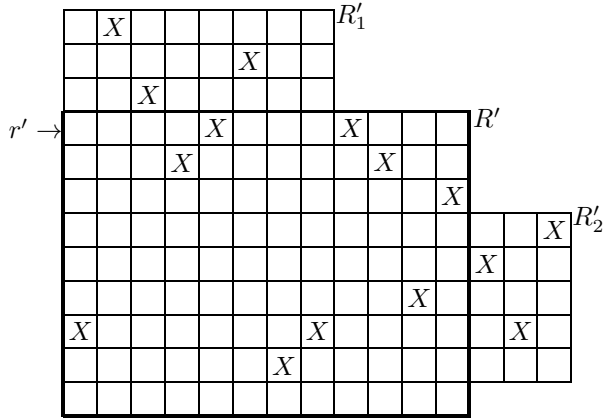

\begin{center}
$$
\Einheit.15cm
\PfadDicke{.3pt}
\Pfad(45,3),222222222222222\endPfad
\Pfad(42,3),222222222222222\endPfad
\Pfad(39,3),222222222222222\endPfad
\Pfad(33,0),222222222222222222222222222\endPfad
\Pfad(30,0),222222222222222222222222222\endPfad
\Pfad(27,0),222222222222222222222222222\endPfad
\Pfad(24,0),222222222222222222222222222222222222\endPfad
\Pfad(21,0),222222222222222222222222222222222222\endPfad
\Pfad(18,0),222222222222222222222222222222222222\endPfad
\Pfad(15,0),222222222222222222222222222222222222\endPfad
\Pfad(12,0),222222222222222222222222222222222222\endPfad
\Pfad(9,0),222222222222222222222222222222222222\endPfad
\Pfad(6,0),222222222222222222222222222222222222\endPfad
\Pfad(3,0),222222222222222222222222222222222222\endPfad
\Pfad(0,0),222222222222222222222222222222222222\endPfad
\Pfad(0,36),111111111111111111111111\endPfad
\Pfad(0,33),111111111111111111111111\endPfad
\Pfad(0,30),111111111111111111111111\endPfad
\Pfad(0,24),111111111111111111111111111111111111\endPfad
\Pfad(0,21),111111111111111111111111111111111111\endPfad
\Pfad(0,18),111111111111111111111111111111111111111111111\endPfad
\Pfad(0,15),111111111111111111111111111111111111111111111\endPfad
\Pfad(0,12),111111111111111111111111111111111111111111111\endPfad
\Pfad(0,9),111111111111111111111111111111111111111111111\endPfad
\Pfad(0,6),111111111111111111111111111111111111111111111\endPfad
\Pfad(0,3),111111111111111111111111111111111111111111111\endPfad
\PfadDicke{1pt}
\Pfad(0,27),111111111111111111111111111111111111\endPfad
\Pfad(0,0),111111111111111111111111111111111111\endPfad
\Pfad(0,0),222222222222222222222222222\endPfad
\Pfad(36,0),222222222222222222222222222\endPfad
\Label\ro{\text {\small$X$}}(1,7)
\Label\ro{\text {\small$X$}}(4,34)
\Label\ro{\text {\small$X$}}(7,28)
\Label\ro{\text {\small$X$}}(10,22)
\Label\ro{\text {\small$X$}}(13,25)
\Label\ro{\text {\small$X$}}(16,31)
\Label\ro{\text {\small$X$}}(19,4)
\Label\ro{\text {\small$X$}}(22,7)
\Label\ro{\text {\small$X$}}(25,25)
\Label\ro{\text {\small$X$}}(28,22)
\Label\ro{\text {\small$X$}}(31,10)
\Label\ro{\text {\small$X$}}(34,19)
\Label\ro{\text {\small$X$}}(37,13)
\Label\ro{\text {\small$X$}}(40,7)
\Label\ro{\text {\small$X$}}(43,16)
\Label\ro{\text {\small$R_1^\prime$}}(25,35)
\Label\ro{\text {\small$R^\prime$}}(37,26)
\Label\ro{\text {\small$R_2^\prime$}}(46,17)
\Label\ro{\text {\small$r^\prime\rightarrow$}}(-3,25)
\hskip5cm
$$
\end{center}
\caption{A filling of the stack polyomino $\mathcal{M}'$.}
\label{fig:12}
\end{figure}

Since the map $\varphi_{R, R'}$ is a  bijection, the map $f$ is as well. The map $\varphi_{R, R'}$ preserves the number of ones, because the words that correspond to the fillings in $R$ and $R'$ under this map have the same recording tableau $Q$, whose number of entries gives the length of the words. Therefore, the map $f$ also preserves the number of ones.

We claim that $f$ preserves the lengths of the longest \emph{ne}- and \emph{se}-chains. To prove this it suffices to show that these statistics are the same when the fillings $M$ and $f(M)$ are restricted to the same maximal rectangle. Recall that the fillings of the rectangles represent words as described before and \emph{ne}-chains (resp. \emph{se}-chains)  correspond to increasing (resp. decreasing) subsequences in those words.

First let $T$ (resp. $T'$) be the filling $M$ (resp. $f(M)$) restricted to the rectangle $R$ (resp. $R$') in $\mathcal{M}$ (resp. $\mathcal{M}'$). By Theorem~\ref{Thm2}, $\mathrm{ne}(T)$ and $\mathrm{se}(T)$ are equal to the number of columns and rows of the Hecke insertion tableau $P$, respectively. Similarly, $\mathrm{ne}(T')$ and $\mathrm{se}(T')$ are equal to the number of columns and rows of the tableau $P'$. Since $P$ and $P'$ have the same shape, we have $\mathrm{ne}(T) = \mathrm{ne}(T')$ and $\mathrm{se}(T) = \mathrm{se}(T')$.

\begin{figure}[h]
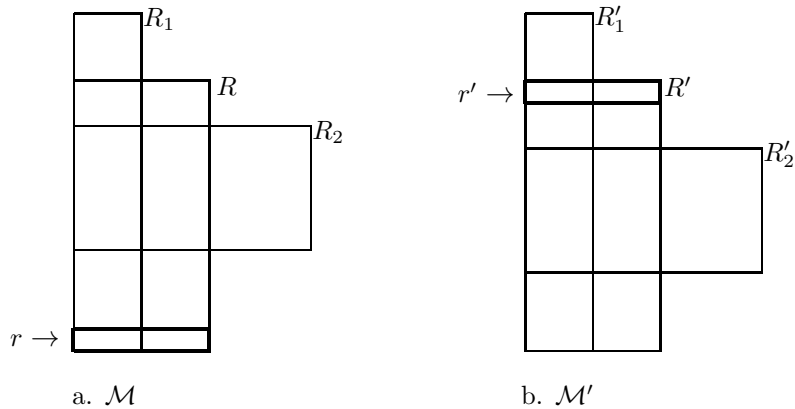

$$
\Einheit.15cm
\PfadDicke{.5pt}
\Pfad(21,9),22222222222\endPfad
\Pfad(12,2),2222222222222222222222\endPfad
\Pfad(6,0),222222222222222222222222222222\endPfad
\Pfad(0,2),2222222222222222222222222222\endPfad
\Pfad(0,30),111111\endPfad
\Pfad(0,24),111111111111\endPfad
\Pfad(0,20),111111111111111111111\endPfad
\Pfad(0,9),111111111111111111111\endPfad
\PfadDicke{1pt}
\Pfad(0,2),111111111111\endPfad
\Pfad(0,0),111111111111\endPfad
\Pfad(12,0),22\endPfad
\Pfad(0,0),22\endPfad
\Label\ro{\text {\small$r$}\rightarrow}(-4,0.5)
\Label\ro{\text {\small$R_1$}}(7,29)
\Label\ro{\text {\small$R$}}(13,23)
\Label\ro{\text {\small$R_2$}}(22,19)
\hbox{\hskip6cm}
\PfadDicke{.5pt}
\Pfad(21,7),22222222222\endPfad
\Pfad(12,0),2222222222222222222222\endPfad
\Pfad(6,0),222222222222222222222222222222\endPfad
\Pfad(0,0),222222222222222222222222222222\endPfad
\Pfad(0,30),111111\endPfad
\Pfad(0,18),111111111111111111111\endPfad
\Pfad(0,7),111111111111111111111\endPfad
\Pfad(0,0),111111111111\endPfad
\PfadDicke{1pt}
\Pfad(0,24),111111111111\endPfad
\Pfad(0,22),111111111111\endPfad
\Pfad(12,22),22\endPfad
\Pfad(0,22),22\endPfad
\Label\ro{\text {\small$r^\prime$}\rightarrow}(-4,22.5)
\Label\ro{\text {\small$R_1^\prime$}}(7,29)
\Label\ro{\text {\small$R^\prime$}}(13,23)
\Label\ro{\text {\small$R_2^\prime$}}(22,17)
\hbox{\hskip3cm}
$$
\centerline{\small 
\hskip2cm
a. $\mathcal{M}$
\hskip5cm
b. $\mathcal{M}'$
\hskip4cm
}
\caption{Moving up the bottom row of a stack polyomino.}
\label{fig:10}
\end{figure}

The rest of the maximal rectangles in $\mathcal{M}$ and $\mathcal{M}'$ fall into two classes: those that are narrower and those that are wider than the rectangle $R$. Let $R_{1}$ be a maximal rectangle in $\mathcal{M}$ that is narrower than $R$  and let $R_{1}'$ be the corresponding rectangle of the same size in $\mathcal{M}'$, as shown in Figure~\ref{fig:10}. We can divide the filling $T_1$ of $R_{1}$ into two parts, $w_1$ (inside $R$) and $\alpha$ (outside of $R$), as in the left filling in Figure~\ref{fig:28}. Similarly, the filling $T_1^\prime$ of $R_1^\prime$, also consists of two parts, $w_1'$ (inside $R'$) and $\alpha$ (outside of $R'$), as in the right filling in Figure~\ref{fig:28}.

\begin{figure}[h]
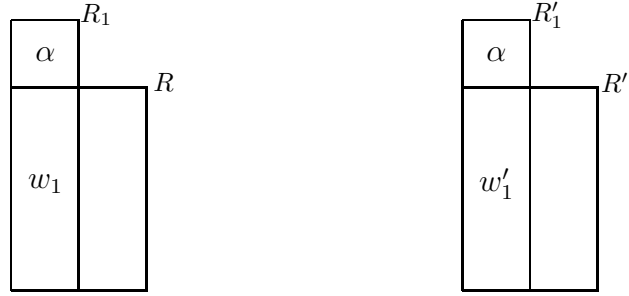

$$
\Einheit.15cm
\PfadDicke{.5pt}
\Pfad(12,0),222222222222222222\endPfad
\Pfad(6,0),222222222222222222222222\endPfad
\Pfad(0,0),222222222222222222222222\endPfad
\Pfad(0,24),111111\endPfad
\Pfad(0,18),111111111111\endPfad
\Pfad(0,0),111111111111\endPfad
\Label\ro{\text {\small$R_1$}}(7,24)
\Label\ro{\text {\small$R$}}(13,18)
\Label\ro{\text {\large$\alpha$}}(2.5,20.5)
\Label\ro{\text {\large$w_1$}}(2.5,9)
\hbox{\hskip6cm}
\PfadDicke{.5pt}
\Pfad(12,0),222222222222222222\endPfad
\Pfad(6,0),222222222222222222222222\endPfad
\Pfad(0,0),222222222222222222222222\endPfad
\Pfad(0,24),111111\endPfad
\Pfad(0,18),111111111111\endPfad
\Pfad(0,0),111111111111\endPfad
\Label\ro{\text {\small$R_1^\prime$}}(7,24)
\Label\ro{\text {\small$R^\prime$}}(13,18)
\Label\ro{\text {\large$\alpha$}}(2.5,20.5)
\Label\ro{\text {\large$w_1^\prime$}}(2.5,9)
\hbox{\hskip3cm}
$$
\caption{Decompositions of the fillings in $R_1$ and $R_1^\prime$.}
\label{fig:28}
\end{figure}

Recall that the growth diagrams of the fillings $T$ and $T'$ are constructed from the bottom left corner and the recording tableau is encoded by the sequence of partitions and edge labels along the top border. Therefore, since the words encoded by $T$ and $T^\prime$ have the same Hecke recording tableau $Q$ by construction,  the recording tableaux $Q(w_1)$ and $Q(w_1^\prime)$ are also the same and are obtained by deleting the values that correspond to the columns outside of the rectangle $R \cap R_{1}$.  Moreover, since the rules for building the growth diagrams are local and the top filling $\alpha$ is the same in both $R_{1}$ and $R_{1}'$, it follows that the words in both $R_{1}$ and $R_{1}'$ have the same recording tableaux as well. Since the Hecke insertion and recording tableaux have the same shape, by Theorem~\ref{Thm2}, we have $ne(T_1)=ne(T_1^\prime)$ and $se(T_1)=se(T_1^\prime)$.

Consider now a maximal rectangle $R_{2}$ in $\mathcal{M}$ which is wider than $R$ and let $R_{2}'$ be the maximal rectangle in $\mathcal{M}'$ of the same size, as in Figure~\ref{fig:15}. Let $T$ and $T_{2}$ be the fillings of $R$ and $R_{2}$ within the filling $M$ of $\mathcal{M}$ and let $T'$ and $T_{2}'$ be the fillings of $R'$ and $R_{2}'$ within the filling $M' = f(M)$ of $\mathcal{M}'$.  These fillings consist of several parts, as represented in Figure~\ref{fig:15}  (some of these parts may be empty). Suppose the filling of $T$, read from bottom up, is the filling of the row $r$, the filling $\gamma$ between row $r$ and $R_{2}$, $w_2$ inside $R \cap R_{2}$, and $\delta$ above $R_{2}$, and let $\beta$ be the filling inside $R_{2}$ but outside of $R$. Similarly, suppose the filling $T^\prime$ of $R'$ consists of $\gamma^\prime$ in the bottom, followed by $w_2^\prime$ inside $R' \cap R_{2}'$, followed by $\delta^\prime$, followed by the filling $r^\prime$ of the row that was moved. The filling  inside $R_{2}'$ but outside of $R'$ is the same $\beta$ by construction.

\begin{figure}[h]
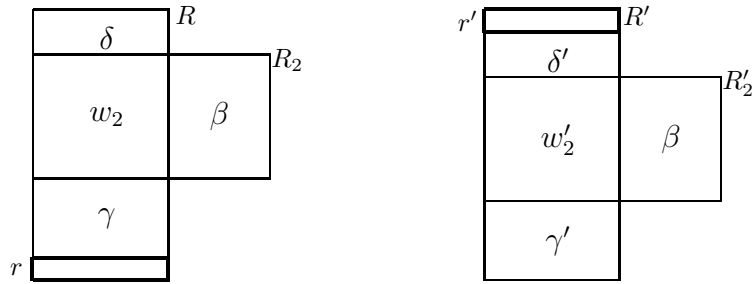

$$
\Einheit.15cm
\PfadDicke{.5pt}
\Pfad(21,9),22222222222\endPfad
\Pfad(12,0),222222222222222222222222\endPfad
\Pfad(0,0),222222222222222222222222\endPfad
\Pfad(0,24),111111111111\endPfad
\Pfad(0,20),111111111111111111111\endPfad
\Pfad(0,9),111111111111111111111\endPfad
\PfadDicke{1pt}
\Pfad(0,2),111111111111\endPfad
\Pfad(0,0),111111111111\endPfad
\Pfad(12,0),22\endPfad
\Pfad(0,0),22\endPfad
\Label\ro{\text {\small$r$}}(-2,0.5)
\Label\ro{\text {\large$\gamma$}}(6,5)
\Label\ro{\text {\large$w_2$}}(6,14)
\Label\ro{\text {\large$\delta$}}(6,21)
\Label\ro{\text {\large$\beta$}}(16,14)
\Label\ro{\text {\small$R$}}(13,23)
\Label\ro{\text {\small$R_{2}$}}(22,19)
\hbox{\hskip6cm}
\PfadDicke{.5pt}
\Pfad(21,7),22222222222\endPfad
\Pfad(12,0),2222222222222222222222\endPfad
\Pfad(0,0),2222222222222222222222\endPfad
\Pfad(0,18),111111111111111111111\endPfad
\Pfad(0,7),111111111111111111111\endPfad
\Pfad(0,0),111111111111\endPfad
\PfadDicke{1pt}
\Pfad(0,24),111111111111\endPfad
\Pfad(0,22),111111111111\endPfad
\Pfad(12,22),22\endPfad
\Pfad(0,22),22\endPfad
\Label\ro{\text {\small$r^\prime$}}(-2,22.5)
\Label\ro{\text {\large$\gamma^\prime$}}(6,3)
\Label\ro{\text {\large$w_2^\prime$}}(6,12)
\Label\ro{\text {\large$\delta^\prime$}}(6,19)
\Label\ro{\text {\large$\beta$}}(16,12)
\Label\ro{\text {\small$R^\prime$}}(13,23)
\Label\ro{\text {\small$R_2^\prime$}}(22,17)
\hbox{\hskip3cm}
$$
\caption{Decompositions of the fillings in $R$, $R_{2}$, $R^\prime$, and $R_2^\prime$.}
\label{fig:15}
\end{figure}

If the row $r$ is empty, then so is $r'$ and we have $\gamma' = \gamma$, $w'_{2} = w_{2}$, $\delta' = \delta$ by construction. So, in this case, $\mathrm{ne}(T_{2}') = \mathrm{ne}(T_{2})$ and $\mathrm{se}(T_{2}') = \mathrm{se}(T_{2})$. 

Now, consider the case when the row $r$ is not empty. For a word $a=a_1a_2\cdots a_r$, let $x_1<x_2<\cdots<x_k$ be the ordered list of letters appearing in $a$. The {\it standardization} of $a$, denoted by $I(a)$, is the word obtained by replacing $x_i$ with $i$ in $a$.  By construction, if $(P, Q)$ is the pair of tableaux that corresponds to $I(w)$ then $(P', Q)$ is the pair of tableaux that correspond to $I(w')$, where $P'$ is constructed from $P$ as explained in the definition of the map $\varphi_{R,R'}$. By Lemma~\ref{Pro3}, we have $I(w) \equiv \mathrm{row}(P)$. Deleting row $r$ corresponds to deleting the 1's in the word $w$, so by Lemma~\ref{Pro1},  we have 
\[ I(w) \backslash 1 \equiv \mathrm{row}(P)\backslash 1 = \mathrm{row}(P/1),\] where $w \backslash 1$ means delete $1$'s from the word $w$ and $P/1$ is the increasing filling of skew shape obtained by deleting the top left corner from $P$. The construction of the tableau $P'$ and Theorem~\ref{Thm3} imply that 
\[I(\mathrm{row}(P/1))\equiv \mathrm{row}(P^\prime/m),\] where $m$ is the maximal number appearing in $P'$. Since $P'$ is the insertion tableau of $I(w')$, by Proposition~\ref{Pro4}, we have that the Hecke insertion tableau of $I(w') \backslash m$ is $P^\prime/m$, which implies 
\[ \mathrm{row}(P^\prime/m)\equiv I(w') \backslash m. \]

Therefore $I(\gamma+w_2+\delta)\equiv I(\gamma^\prime+w_2^\prime+\delta^\prime)$.
Since $\gamma$ (resp. $\delta$) and $\gamma'$ (resp. $\delta'$) have the same number of nonempty rows, by using Lemma~\ref{Pro1} again, we have $I(w_2)\equiv I(w_2^\prime)$. Since the $K$-Knuth relations consist of local transformations, this implies $I(w_2\beta) \equiv I(w_2^\prime\beta)$.  By Proposition~\ref{Pro2}, we have $\mathrm{ne}(T_2)=\mathrm{ne}(T_2^\prime)$ and $\mathrm{se}(T_2)=\mathrm{se}(T_2^\prime)$.

\end{proof}

\section{Special case: Ferrers shapes}\label{S:FS}


In this section, we explain how Theorem~\ref{Thm1} gives a different proof of the main result in~\cite{CGP}.

\textbf{1.} The objects of consideration in \cite{CGP} are linked partitions of $[n]=\{1, 2, \ldots, n\}$.
A {\it linked partition} of $[n]$ is a collection of nonempty subsets $B_1$, $B_2$, $\ldots$, $B_k$ of $[n]$,
called $blocks$, such that $\cup_{i=1}^{k} B_i = [n]$ and any two distinct blocks
are nearly disjoint. Two distinct blocks $B_i$ and $B_j$ are said to be {\it nearly disjoint} if 
for any $t\in B_i\cap B_j$, one of the following conditions holds:
\begin{itemize}
\item[(1)] $t=\min(B_i)$, $|B_i|>1$, and $t\neq \min(B_j)$,
\item[(2)] $t=\min(B_j)$, $|B_j|>1$, and $t\neq \min(B_i)$.
\end{itemize}

Given a linked partition $P$ of $[n]$, 
a block $\{i_1,i_2,\ldots,i_m\}$, $i_1<i_2<\cdots<i_m$ of $P$ is represented
by the set of pairs $\{(i_1,i_2), (i_1,i_3), \ldots, (i_1, i_m)\}$.
More generally, a linked partition is represented by the union of 
all set of pairs, the union being taken over all its blocks.
This representation is called the {\it standard representation} of the linked partition.
For example, the linked partition $\{\{1,3,6\},\{2,5,8\},\{4\},\{5,9\},\{6,7\}\}$
is represented as the set $\{(1,3),(1,6),(2,5),(2,8),(5,9),(6,7)\}$.
Next, one defines a {\it $k$-crossing} of a linked partition to be a subset
$\{(i_1,j_1),(i_2,j_2),\dots,(i_k,j_k)\}$ of
its standard representation where $i_1<i_2<\dots<i_k<j_1<j_2<\dots<j_k$. 
Similarly, one defines a {\it $k$-nesting} of a linked partition 
to be a subset $\{(i_1,j_1),(i_2,j_2),\dots,(i_k,j_k)\}$ of
its standard representation where
$i_1<i_2<\dots<i_k<j_k<\dots<j_2<j_1$. (These notions have 
intuitive pictorial meanings if one connects a pair $(i,j)$ in the
standard representation of a linked partition by an arc, cf.\ \cite{CGP}.)
Finally, given a linked partition $P$, we write $\mathrm{cross}(P)$ for the maximal
number $k$ such that $P$ has a $k$-crossing, and we write $\mathrm{nest}(P)$ 
for the maximal number $k$ such that $P$ has a $k$-nesting.

Given a linked partition $P$, let $\mathrm{comp1}(P)$ be the set of the first components 
of the pairs of its standard representation,
and let $\mathrm{comp2}(P)$ be the set of the second components of the pairs of its standard representation.
Then Theorem~4.2 from \cite{CGP} reads as follows.

\begin{thm}[\cite{CGP}] \label{Thm5} 
Let $n,x,y$ be positive integers, and let $S$ and $T$ be two subsets of $[n]$. 
Then the number of linked partitions $P$ of $[n]$ with $\mathrm{cross}(P)=x$,
$\mathrm{nest}(P)=y$, $\mathrm{comp1}(P)=S$, $\mathrm{comp2}(P)=T$ is equal to the 
number of linked partitions of $[n]$ with $\mathrm{cross}(P)=y$,
$\mathrm{nest}(P)=x$, $\mathrm{comp1}(P)=S$, $\mathrm{comp2}(P)=T$.
\end{thm}

There is a one-to-one correspondence between linked partitions on $[n]$
and fillings of $\vartriangle_n$, the triangular shape with 
$n-1$ squares in the bottom row, $n-2$ squares in the row above, etc., and 
$1$ square in the top-most row.  We represent a linked partition $P$ of $[n]$, given by its standard representation,
as a filling, by putting an $X$ to the square in row $i$ (from top to bottom) and column $j$ (counted from right to left, including one empty column) if and only if $(i,j)$ is a pair in the standard representation of $P$. See Figure~\ref{fig:18} for two examples in which $n=7$.  (The labeling of the corners and edges of the fillings should be ignored at this point.)  This  defines a correspondence between linked partitions of $[n]$ and 01-fillings of $\vartriangle_n$ with at most one 1 in each column. Moreover, a $k$-crossing of $P$ corresponds to a $se$-chain of length $k$ in $\vartriangle_n$, and a $k$-nesting of $P$ corresponds to a $ne$-chain of length $k$ in $\vartriangle_n$.

\begin{figure}[h]
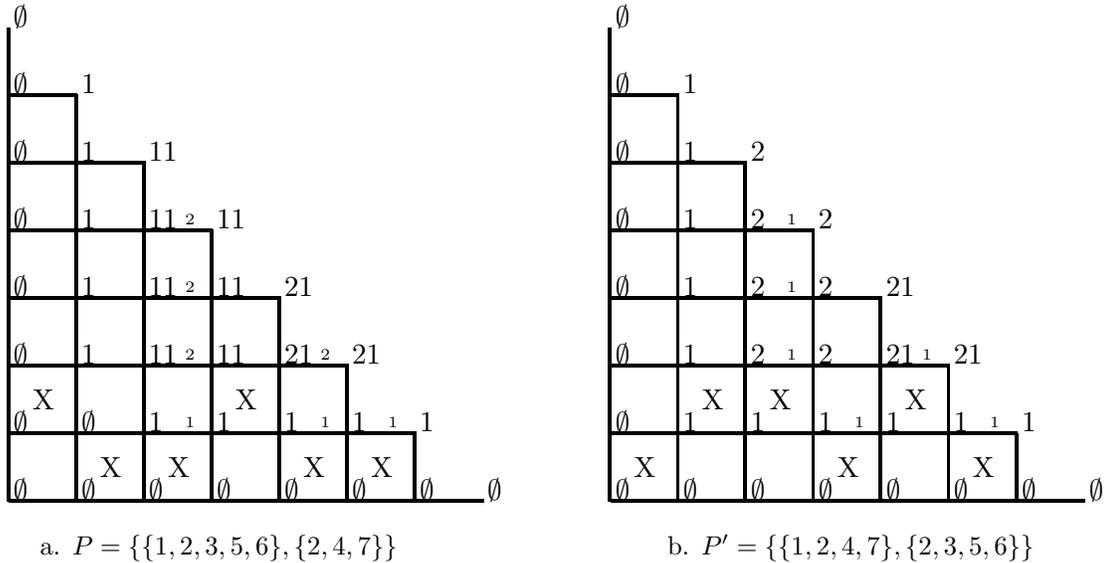

$$
\Einheit.3cm
\Pfad(0,21),666666666666666666666111111111111111111111\endPfad
\Pfad(0,18),111666666666666666666\endPfad
\Pfad(0,15),111111666666666666666\endPfad
\Pfad(0,12),111111111666666666666\endPfad
\Pfad(0,9),111111111111666666666\endPfad
\Pfad(0,6),111111111111111666666\endPfad
\Pfad(0,3),111111111111111111666\endPfad
\Label\ro{\emptyset}(0,0)
\Label\ro{\emptyset}(0,3)
\Label\ro{\emptyset}(0,6)
\Label\ro{\emptyset}(0,9)
\Label\ro{\emptyset}(0,12)
\Label\ro{\emptyset}(0,15)
\Label\ro{\emptyset}(0,18)
\Label\ro{\emptyset}(0,21)
\Label\ro{1}(3,18)
\Label\ro{1}(3,15)
\Label\ro{1}(3,12)
\Label\ro{1}(3,9)
\Label\ro{1}(3,6)
\Label\ro{\emptyset}(3,3)
\Label\ro{\emptyset}(3,0)
\Label\ro{\hphantom{1}11}(6,15)
\Label\ro{\hphantom{1}11}(6,12)
\Label\ro{\hphantom{1}11}(6,9)
\Label\ro{\hphantom{1}11}(6,6)
\Label\ro{1}(6,3)
\Label\ro{\emptyset}(6,0)
\Label\ro{\hphantom{1}11}(9,12)
\Label\ro{\hphantom{1}11}(9,9)
\Label\ro{\hphantom{1}11}(9,6)
\Label\ro{1}(9,3)
\Label\ro{\emptyset}(9,0)
\Label\ro{\hphantom{1}21}(12,9)
\Label\ro{\hphantom{1}21}(12,6)
\Label\ro{1}(12,3)
\Label\ro{\emptyset}(12,0)
\Label\ro{\hphantom{2}21}(15,6)
\Label\ro{1}(15,3)
\Label\ro{\emptyset}(15,0)
\Label\ro{1}(18,3)
\Label\ro{\emptyset}(18,0)
\Label\ro{\emptyset}(21,0)
\Label\ro{\tiny\text{1}}(7.5,3)
\Label\ro{\tiny\text{2}}(7.5,6)
\Label\ro{\tiny\text{2}}(7.5,9)
\Label\ro{\tiny\text{2}}(7.5,12)
\Label\ro{\tiny\text{1}}(13.5,3)
\Label\ro{\tiny\text{2}}(13.5,6)
\Label\ro{\tiny\text{1}}(16.5,3)
\Label\ro{\text {X}}(1,4)
\Label\ro{\text {X}}(4,1)
\Label\ro{\text {X}}(7,1)
\Label\ro{\text {X}}(10,4)
\Label\ro{\text {X}}(13,1)
\Label\ro{\text {X}}(16,1)
\hbox{\hskip8cm}
\Pfad(0,21),666666666666666666666111111111111111111111\endPfad
\Pfad(0,18),111666666666666666666\endPfad
\Pfad(0,15),111111666666666666666\endPfad
\Pfad(0,12),111111111666666666666\endPfad
\Pfad(0,9),111111111111666666666\endPfad
\Pfad(0,6),111111111111111666666\endPfad
\Pfad(0,3),111111111111111111666\endPfad
\Label\ro{\emptyset}(0,0)
\Label\ro{\emptyset}(0,3)
\Label\ro{\emptyset}(0,6)
\Label\ro{\emptyset}(0,9)
\Label\ro{\emptyset}(0,12)
\Label\ro{\emptyset}(0,15)
\Label\ro{\emptyset}(0,18)
\Label\ro{\emptyset}(0,21)
\Label\ro{1}(3,18)
\Label\ro{1}(3,15)
\Label\ro{1}(3,12)
\Label\ro{1}(3,9)
\Label\ro{1}(3,6)
\Label\ro{1}(3,3)
\Label\ro{\emptyset}(3,0)
\Label\ro{2}(6,15)
\Label\ro{2}(6,12)
\Label\ro{2}(6,9)
\Label\ro{2}(6,6)
\Label\ro{1}(6,3)
\Label\ro{\emptyset}(6,0)
\Label\ro{2}(9,12)
\Label\ro{2}(9,9)
\Label\ro{2}(9,6)
\Label\ro{1}(9,3)
\Label\ro{\emptyset}(9,0)
\Label\ro{\hphantom{1}21}(12,9)
\Label\ro{\hphantom{1}21}(12,6)
\Label\ro{1}(12,3)
\Label\ro{\emptyset}(12,0)
\Label\ro{\hphantom{2}21}(15,6)
\Label\ro{1}(15,3)
\Label\ro{\emptyset}(15,0)
\Label\ro{1}(18,3)
\Label\ro{\emptyset}(18,0)
\Label\ro{\emptyset}(21,0)
\Label\ro{\tiny\text{1}}(7.5,6)
\Label\ro{\tiny\text{1}}(7.5,9)
\Label\ro{\tiny\text{1}}(7.5,12)
\Label\ro{\tiny\text{1}}(10.5,3)
\Label\ro{\tiny\text{1}}(13.5,6)
\Label\ro{\tiny\text{1}}(16.5,3)
\Label\ro{\text {X}}(1,1)
\Label\ro{\text {X}}(4,4)
\Label\ro{\text {X}}(7,4)
\Label\ro{\text {X}}(10,1)
\Label\ro{\text {X}}(13,4)
\Label\ro{\text {X}}(16,1)
\hskip6cm
$$
\centerline{\small a. $P=\{\{1,2,3,5,6\},\{2,4,7\}\}$
\hskip3.5cm
b. $P^\prime=\{\{1,2,4,7\},\{2,3,5,6\}\}$}
\hskip1cm
\caption{Bijection on linked partitions using Hecke growth diagrams.} 
\label{fig:18}
\end{figure}

If we specialize Theorem~\ref{Thm1} to the case where $\mathcal{M}=\vartriangle_n$
and $\sigma\mathcal{M}=\triangledown_n$ is a reflection of $\mathcal{M}$ through a horizontal line,
we obtain \[N(\vartriangle_n;n;ne=u,se=v)=N(\triangledown_n;n;ne=u,se=v).\]
On the other hand reflecting the polyomino with a filling about a horizontal line exchanges the \emph{se}- and \emph{ne}-chains, so we have
\[ N(\triangledown_n;n;ne=u,se=v) =  N(\vartriangle_n;n;ne=v,se=u). \]

The set of the first (resp. second) components of the linked partition correspond to the nonempty rows (resp. columns). Since our bijection $f$ in Section~\ref{S:proof} and the reflection both preserve the empty rows and columns, we obtain Theorem~\ref{Thm5}.

\textbf{2.} The proof of Theorem~\ref{Thm5} in~\cite{CGP} goes via a bijection $\phi$ between linked partitions and \emph{vacillating Hecke tableaux} of empty shape. These are certain sequences of Hecke diagrams, which are Young diagrams with a possibly marked corner. For a linked partition $P$,  the corresponding vacillating Hecke tableaux $\phi(P)$ is obtained via Hecke insertion. The maximum number of columns and rows in the Hecke diagrams in $\phi(P)$ is equal to $\mathrm{ness}(P)$ and $\mathrm{cross}(P)$. Transposing a vacillating Hecke tableaux of empty shape yields a vacillating Hecke tableaux of empty shape, so applying $\phi^{-1}$ yields a bijection on linked partitions which preserves the first and second components, which proves Theorem~\ref{Thm5}.

Though we omit the full definition of vacillating Hecke tableaux here, we note that even though the authors in \cite{CGP} do not use the language of growth diagrams, $\phi(P)$ is equivalent to the sequence of labels of the vertices and the horizontal edges along the top right border of the Hecke growth diagram of $P$ obtained as explained in Section~\ref{S:tools}.  The sequence should be read from bottom right to top left. The labels of the horizontal edges indicate the rows which have marked corners.  For example, the vacillating Hecke tableaux $\phi(P)$ for the linked partition $P$ in Figure~\ref{fig:18}a is given in Figure~\ref{fig:vactab}. This matches the example shown in Figure 12 in~\cite{CGP}. Transposing the vacillating Hecke tableau corresponds to transposing the partitions along the top right border of $\vartriangle_n$ and changing the labels of the horizontal edges along that border to indicate the column in which the marked corner appears. Then by applying the reverse local growth rules, we obtain a filling which represents a linked partition, like in Figure~\ref{fig:18}b.

Note that this description of the map in~\cite{CGP} in terms of Hecke growth diagrams is analogous to the description of the map from~\cite{CDDSY} between set partitions and vacillating tableaux that Krattenthaler~\cite{Krattenthaler} gave in terms of Fomin's growth diagrams for RSK.

\begin{figure}[h]
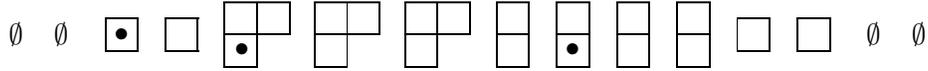

\begin{center}
$$
\hspace*{3cm} 
\Einheit.4cm
\Label\ro{\emptyset}(10.5,0)
\Label\ro{\emptyset}(9,0)
\Label\ro{\yng(1)}(7,0)
\Label\ro{\yng(1)}(5,0)
\Label\ro{\yng(1,1)}(3,0)
\Label\ro{\yng(1,1)}(1,0)
\Label\ro{\yng(1,1)}(-1,0)
\Label\ro{ \small{\bullet}}(-1,-0.5)
\Label\ro{\yng(1,1)}(-3,0)
\Label\ro{\yng(2,1)}(-5.5,0)
\Label\ro{\yng(2,1)}(-8.5,0)
\Label\ro{\yng(2,1)}(-11.5,0)
\Label\ro{\small{\bullet}}(-12,-0.5)
\Label\ro{\yng(1)}(-14,0)
\Label\ro{\yng(1)}(-16,0)
\Label\ro{\bullet}(-16,0)
\Label\ro{\emptyset}(-18,0)
\Label\ro{\emptyset}(-19.5,0)
$$
\end{center}
\caption{A vacillating Hecke tableau for the linked partition $P = \{\{1,2,3,5,6\},\{2,4,7\}\}$.}
\label{fig:vactab}
\end{figure}

Since our map $f$ from Section~\ref{S:proof} preserves the 1's outside the maximal rectangle that contains the main row, the bijection on linked partitions that we get by successively applying $f$ between the fillings of $\vartriangle_n$ and $\triangledown_n$ is different from the one in~\cite{CGP}. Indeed, consider the linked partition:
$$L = \{\{1,2,4,7,10\},\{2,3,6,15\},\{3,11\},\{4,5,8\},\{5,12,14,16\},\{6,9\},\{7,13\}\}.$$
The corresponding linked partition that we get via our bijection is
$$L_{1} = \{\{1,2,3,8,9,13\},\{2,5,14\},\{3,4,6\},\{4,12,16\},\{5,7,11\},\{6,15\},\{7,10\}\},$$
while the bijection in~\cite{CGP} produces 
$$L_{2}=\{\{1,2,3,13\},\{2,5,14\},\{3,4,6\},\{4,12,16\},\{5,7,8,9\},\{6,11,15\},\{7,10\}\}.$$

\section{Discussion}\label{S:discussion}


In this section we comment on possible and impossible extensions of the results in this paper.

\textbf{1.} As observed in~\cite{MR2274298}, Theorem~\ref{Thm1} doesn't hold if we fix the number of 1's in each row. Let $N^{(a_{1},a_{2}, \cdots)}(\mathrm{F}; n; ne = u, se = v)$ be the number of fillings of the Ferrers shape $\mathrm{F}$ (in French notation) such that: the sum of entries equal to $n$, the longest \emph{ne}-chain has length $u$, the longest \emph{se}-chain has length $v$, and with $a_{i}$ 1's in row $i$, $i = 1, 2, \dots$. Then for $\mathrm{F} = (7,7,6,4)$, we have $N^{(4,1,1,1)}(\mathrm{F}; n = 7; ne = 4, se = 1) = 0$, but $N^{(4,1,1,1)}(\mathrm{F}; n = 7; ne = 1, se = 4) = 1$. 

\textbf{2.} In Section~\ref{S:proof}, we showed
\begin{equation*} \label{eq1}
N(\mathcal{M};n;ne=u, se=v)=N(\mathcal{M}';n;ne=u, se=v)
\end{equation*}
where $\mathcal{M}'$ is a stack polyomino obtained by moving the {\it bottom} row of 
$\mathcal{M}$ up as much as possible, by using the transformation $\varphi_{R, R'}$ within the maximal rectangle containing the bottom row, while the rest of the filling remained the same. One may ask whether one can use the same transformation $\varphi$ on fillings of rectangles to move an arbitrary row of a stack polyomino up 
(such that the result is again a stack polyomino) so that the $\mathrm{ne}$ and $\mathrm{se}$ statistics are preserved. 

The answer to this question is negative. In the following we give a counterexample.
Figure~\ref{fig:19} shows a filling $F$ of a stack polyomino $\mathcal{M}$ with $ne=5$ and $se=6$.
Moving the row $r$ (third from below) of $\mathcal{M}$ to the top produces the stack polyomino $\mathcal{M^\prime}$ in 
Figure~\ref{fig:20}. Let $R$ (resp. $R^\prime$) be the largest rectangle contained in $\mathcal{M}$ (resp. $\mathcal{M^\prime}$)
that contains $r$ (resp. $r^\prime$).
By applying the bijection $\varphi_{R,R^\prime}$ (defined in Section~\ref{S:proof})
between the filling of $R$ and the filling of $R^\prime$,
while all entries out of $R'$ stay the same, one obtains the filling of $\mathcal{M^\prime}$ 
in Figure~\ref{fig:20}, for which $ne=4$ and $se=6$.

\begin{figure}[h]
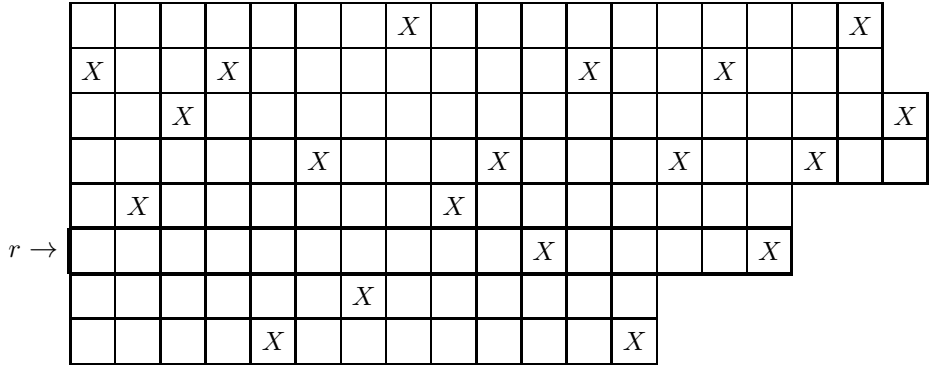

$$
\Einheit.2cm
\PfadDicke{.5pt}
\Pfad(0,24),111111111111111111111111111111111111111111111111111111\endPfad
\Pfad(0,21),111111111111111111111111111111111111111111111111111111\endPfad
\Pfad(0,18),111111111111111111111111111111111111111111111111111111111\endPfad
\Pfad(0,15),111111111111111111111111111111111111111111111111111111111\endPfad
\Pfad(0,12),111111111111111111111111111111111111111111111111111111111\endPfad
\Pfad(0,3),111111111111111111111111111111111111111\endPfad
\Pfad(0,0),111111111111111111111111111111111111111\endPfad
\Pfad(0,0),222222222222222222222222\endPfad
\Pfad(3,0),222222222222222222222222\endPfad
\Pfad(6,0),222222222222222222222222\endPfad
\Pfad(9,0),222222222222222222222222\endPfad
\Pfad(12,0),222222222222222222222222\endPfad
\Pfad(15,0),222222222222222222222222\endPfad
\Pfad(18,0),222222222222222222222222\endPfad
\Pfad(21,0),222222222222222222222222\endPfad
\Pfad(24,0),222222222222222222222222\endPfad
\Pfad(27,0),222222222222222222222222\endPfad
\Pfad(30,0),222222222222222222222222\endPfad
\Pfad(33,0),222222222222222222222222\endPfad
\Pfad(36,0),222222222222222222222222\endPfad
\Pfad(39,0),222222222222222222222222\endPfad
\Pfad(42,6),222222222222222222\endPfad
\Pfad(45,6),222222222222222222\endPfad
\Pfad(48,6),222222222222222222\endPfad
\Pfad(51,12),222222222222\endPfad
\Pfad(54,12),222222222222\endPfad
\Pfad(57,12),222222\endPfad
\PfadDicke{1pt}
\Pfad(0,9),111111111111111111111111111111111111111111111111\endPfad
\Pfad(0,6),111111111111111111111111111111111111111111111111\endPfad
\Pfad(0,6),222\endPfad
\Pfad(48,6),222\endPfad
\Label\ro{\text {\small$X$}}(1,19)
\Label\ro{\text {\small$X$}}(4,10)
\Label\ro{\text {\small$X$}}(7,16)
\Label\ro{\text {\small$X$}}(10,19)
\Label\ro{\text {\small$X$}}(13,1)
\Label\ro{\text {\small$X$}}(16,13)
\Label\ro{\text {\small$X$}}(19,4)
\Label\ro{\text {\small$X$}}(22,22)
\Label\ro{\text {\small$X$}}(25,10)
\Label\ro{\text {\small$X$}}(28,13)
\Label\ro{\text {\small$X$}}(31,7)
\Label\ro{\text {\small$X$}}(34,19)
\Label\ro{\text {\small$X$}}(37,1)
\Label\ro{\text {\small$X$}}(40,13)
\Label\ro{\text {\small$X$}}(43,19)
\Label\ro{\text {\small$X$}}(46,7)
\Label\ro{\text {\small$X$}}(49,13)
\Label\ro{\text {\small$X$}}(52,22)
\Label\ro{\text {\small$X$}}(55,16)
\Label\ro{\text {\small$r$}\rightarrow}(-3,7)
\hskip12cm
$$
\caption{A filling of $\mathcal{M}$ with $ne=5,\,se=6$.}
\label{fig:19}
\end{figure}

\begin{figure}[t]
$$
\Einheit.2cm
\PfadDicke{.5pt}
\Pfad(48,21),111111\endPfad
\Pfad(0,18),111111111111111111111111111111111111111111111111111111\endPfad
\Pfad(0,15),111111111111111111111111111111111111111111111111111111111\endPfad
\Pfad(0,12),111111111111111111111111111111111111111111111111111111111\endPfad
\Pfad(0,9),111111111111111111111111111111111111111111111111111111111\endPfad
\Pfad(0,6),111111111111111111111111111111111111111111111111\endPfad
\Pfad(0,3),111111111111111111111111111111111111111\endPfad
\Pfad(0,0),111111111111111111111111111111111111111\endPfad
\Pfad(0,0),222222222222222222222222\endPfad
\Pfad(3,0),222222222222222222222222\endPfad
\Pfad(6,0),222222222222222222222222\endPfad
\Pfad(9,0),222222222222222222222222\endPfad
\Pfad(12,0),222222222222222222222222\endPfad
\Pfad(15,0),222222222222222222222222\endPfad
\Pfad(18,0),222222222222222222222222\endPfad
\Pfad(21,0),222222222222222222222222\endPfad
\Pfad(24,0),222222222222222222222222\endPfad
\Pfad(27,0),222222222222222222222222\endPfad
\Pfad(30,0),222222222222222222222222\endPfad
\Pfad(33,0),222222222222222222222222\endPfad
\Pfad(36,0),222222222222222222222222\endPfad
\Pfad(39,0),222222222222222222222222\endPfad
\Pfad(42,6),222222222222222222\endPfad
\Pfad(45,6),222222222222222222\endPfad
\Pfad(48,6),222222222222222222\endPfad
\Pfad(51,9),222222222222\endPfad
\Pfad(54,9),222222222222\endPfad
\Pfad(57,9),222222\endPfad
\PfadDicke{1pt}
\Pfad(0,24),111111111111111111111111111111111111111111111111\endPfad
\Pfad(0,21),111111111111111111111111111111111111111111111111\endPfad
\Pfad(0,21),222\endPfad
\Pfad(48,21),222\endPfad
\Label\ro{\text {\small$X$}}(1,22)
\Label\ro{\text {\small$X$}}(4,7)
\Label\ro{\text {\small$X$}}(7,13)
\Label\ro{\text {\small$X$}}(10,16)
\Label\ro{\text {\small$X$}}(13,1)
\Label\ro{\text {\small$X$}}(16,13)
\Label\ro{\text {\small$X$}}(19,4)
\Label\ro{\text {\small$X$}}(22,19)
\Label\ro{\text {\small$X$}}(25,10)
\Label\ro{\text {\small$X$}}(28,10)
\Label\ro{\text {\small$X$}}(31,7)
\Label\ro{\text {\small$X$}}(34,16)
\Label\ro{\text {\small$X$}}(37,1)
\Label\ro{\text {\small$X$}}(40,10)
\Label\ro{\text {\small$X$}}(43,16)
\Label\ro{\text {\small$X$}}(46,7)
\Label\ro{\text {\small$X$}}(49,10)
\Label\ro{\text {\small$X$}}(52,19)
\Label\ro{\text {\small$X$}}(55,13)
\Label\ro{\text {\small$r^\prime$}\rightarrow}(-3,22)
\hskip12cm
$$
\caption{A filling of $\mathcal{M^\prime}$ with $ne=4,\,se=6$.}
\label{fig:20}
\end{figure}

\textbf{3.} With a help of a computer program, we computed $\sum_{M\in\mathcal{F}(\mathcal{M})} x^{\ne(M) } y^{\se(M)}$ for different moon polyominoes $\mathcal{M}$ and, based on these calculations, we make the following conjecture.

\begin{conj} Let $\mathcal{M}$ be a moon polyomino. The number $N(\mathcal{M};n;ne=u;se=v)$ only depends on the lengths of the rows of $\mathcal{M}$, not the actual shape. In other words, if $\mathcal{M}_{1}$ and $\mathcal{M}_{2}$ are  obtained by permuting the rows and columns of $\mathcal{M}$, respectively, then
\[ N(\mathcal{M};n;ne=u;se=v) = N(\mathcal{M}_{1};n;ne=u;se=v) = N(\mathcal{M}_{2};n;ne=u;se=v).\]
\end{conj}

For example, let $\mathcal{M}_1$, $\mathcal{M}_2$, $\mathcal{M}_3$ and $\mathcal{M}_4$ be the moon polyominoes given in Figure~\ref{fig:22}.
For $1\leq i \leq 4$, let $G_i(x,y)=\sum_{M\in\mathcal{F}(\mathcal{M}_i)} x^{\ne(M) } y^{\se(M)}$
be the joint distribution of $(ne,se)$ over fillings of $\mathcal{M}_i^\prime$. Then
\begin{align*}  
G_1(x,y) &=G_2(x,y)=G_3(x,y)=G_4(x,y)\nonumber \\ 
&=40x^3y^3+238(x^3y^2+x^2y^3)+4(x^3y+xy^3)+348x^2y^2+2(x^2y+xy^2).
\end{align*}

\begin{figure}[b]
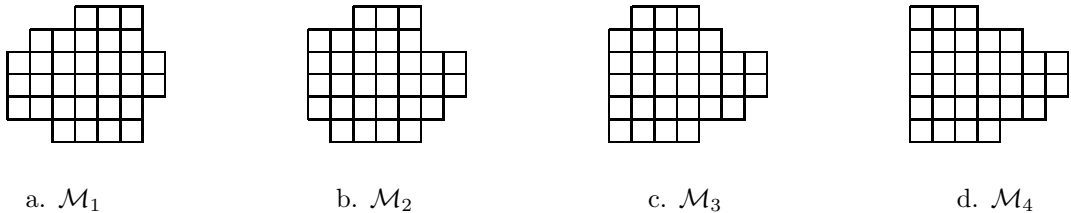

$$
\Einheit.1cm
\PfadDicke{.5pt}
\Pfad(0,3),222222222\endPfad
\Pfad(3,3),222222222222\endPfad
\Pfad(6,0),222222222222222\endPfad
\Pfad(9,0),222222222222222222\endPfad
\Pfad(12,0),222222222222222222\endPfad
\Pfad(15,0),222222222222222222\endPfad
\Pfad(18,0),222222222222222222\endPfad
\Pfad(21,6),222222\endPfad
\Pfad(9,18),111111111\endPfad
\Pfad(3,15),111111111111111\endPfad
\Pfad(0,12),111111111111111111111\endPfad
\Pfad(0,9),111111111111111111111\endPfad
\Pfad(0,6),111111111111111111111\endPfad
\Pfad(0,3),111111111111111111\endPfad
\Pfad(6,0),111111111111\endPfad
\hbox{\hskip4cm}
\PfadDicke{.5pt}
\Pfad(18,3),222222222\endPfad
\Pfad(0,3),222222222222\endPfad
\Pfad(3,0),222222222222222\endPfad
\Pfad(6,0),222222222222222222\endPfad
\Pfad(9,0),222222222222222222\endPfad
\Pfad(12,0),222222222222222222\endPfad
\Pfad(15,0),222222222222222222\endPfad
\Pfad(21,6),222222\endPfad
\Pfad(6,18),111111111\endPfad
\Pfad(0,15),111111111111111\endPfad
\Pfad(0,12),111111111111111111111\endPfad
\Pfad(0,9),111111111111111111111\endPfad
\Pfad(0,6),111111111111111111111\endPfad
\Pfad(0,3),111111111111111111\endPfad
\Pfad(3,0),111111111111\endPfad
\hbox{\hskip4cm}
\PfadDicke{.5pt}
\Pfad(18,3),222222222\endPfad
\Pfad(15,3),222222222222\endPfad
\Pfad(0,0),222222222222222\endPfad
\Pfad(3,0),222222222222222222\endPfad
\Pfad(6,0),222222222222222222\endPfad
\Pfad(9,0),222222222222222222\endPfad
\Pfad(12,0),222222222222222222\endPfad
\Pfad(21,6),222222\endPfad
\Pfad(3,18),111111111\endPfad
\Pfad(0,15),111111111111111\endPfad
\Pfad(0,12),111111111111111111111\endPfad
\Pfad(0,9),111111111111111111111\endPfad
\Pfad(0,6),111111111111111111111\endPfad
\Pfad(0,3),111111111111111111\endPfad
\Pfad(0,0),111111111111\endPfad
\hbox{\hskip4cm}
\PfadDicke{.5pt}
\Pfad(18,3),222222222\endPfad
\Pfad(15,3),222222222222\endPfad
\Pfad(12,0),222222222222222\endPfad
\Pfad(0,0),222222222222222222\endPfad
\Pfad(3,0),222222222222222222\endPfad
\Pfad(6,0),222222222222222222\endPfad
\Pfad(9,0),222222222222222222\endPfad
\Pfad(21,6),222222\endPfad
\Pfad(0,18),111111111\endPfad
\Pfad(0,15),111111111111111\endPfad
\Pfad(0,12),111111111111111111111\endPfad
\Pfad(0,9),111111111111111111111\endPfad
\Pfad(0,6),111111111111111111111\endPfad
\Pfad(0,3),111111111111111111\endPfad
\Pfad(0,0),111111111111\endPfad
\hbox{\hskip3cm}
$$
\centerline{\small a. $\mathcal{M}_1$
\hskip3cm
b. $\mathcal{M}_2$
\hskip3cm
c. $\mathcal{M}_3$
\hskip3cm
d. $\mathcal{M}_4$
\hskip1cm}
\caption{Moon polyominoes with same row lengths.}
\label{fig:22}
\end{figure}

However, this conjecture cannot be proved using the same map from Section~\ref{S:proof} which moves the bottom row $r$ up while everything outside the maximal rectangle which contains $r$ remains the same. Namely, applying this idea to the filling in Figure~\ref{fig:24} which has $ne=5,\,se=5$ yields the filling in Figure~\ref{fig:25} with $ne=4,\,se=5$.

\begin{figure}[h]
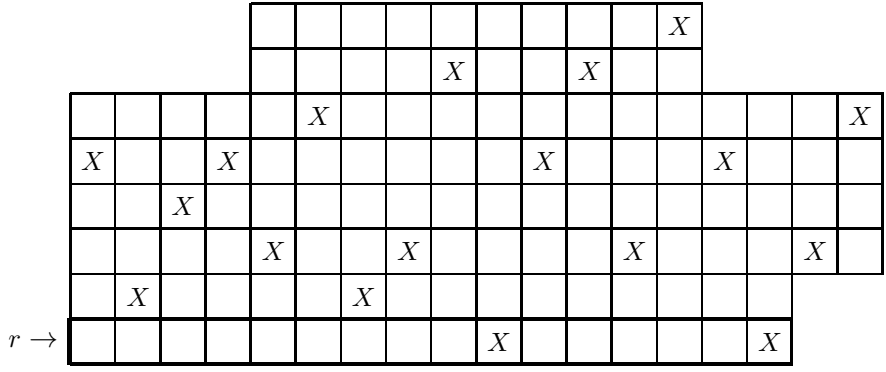

$$
\Einheit.2cm
\PfadDicke{.5pt}
\Pfad(0,18),111111111111111111111111111111111111111111111111111111\endPfad
\Pfad(0,15),111111111111111111111111111111111111111111111111111111\endPfad
\Pfad(0,12),111111111111111111111111111111111111111111111111111111\endPfad
\Pfad(0,9),111111111111111111111111111111111111111111111111111111\endPfad
\Pfad(0,6),111111111111111111111111111111111111111111111111111111\endPfad
\Pfad(12,24),111111111111111111111111111111\endPfad
\Pfad(12,21),111111111111111111111111111111\endPfad
\Pfad(0,0),222222222222222222\endPfad
\Pfad(3,0),222222222222222222\endPfad
\Pfad(6,0),222222222222222222\endPfad
\Pfad(9,0),222222222222222222\endPfad
\Pfad(12,0),222222222222222222222222\endPfad
\Pfad(15,0),222222222222222222222222\endPfad
\Pfad(18,0),222222222222222222222222\endPfad
\Pfad(21,0),222222222222222222222222\endPfad
\Pfad(24,0),222222222222222222222222\endPfad
\Pfad(27,0),222222222222222222222222\endPfad
\Pfad(30,0),222222222222222222222222\endPfad
\Pfad(33,0),222222222222222222222222\endPfad
\Pfad(36,0),222222222222222222222222\endPfad
\Pfad(39,0),222222222222222222222222\endPfad
\Pfad(42,0),222222222222222222222222\endPfad
\Pfad(45,0),222222222222222222\endPfad
\Pfad(48,0),222222222222222222\endPfad
\Pfad(51,6),222222222222\endPfad
\Pfad(54,6),222222222222\endPfad
\PfadDicke{1pt}
\Pfad(0,3),111111111111111111111111111111111111111111111111\endPfad
\Pfad(0,0),111111111111111111111111111111111111111111111111\endPfad
\Pfad(0,0),222\endPfad
\Pfad(48,0),222\endPfad
\Label\ro{\text {\small$X$}}(1,13)
\Label\ro{\text {\small$X$}}(4,4)
\Label\ro{\text {\small$X$}}(7,10)
\Label\ro{\text {\small$X$}}(10,13)
\Label\ro{\text {\small$X$}}(13,7)
\Label\ro{\text {\small$X$}}(16,16)
\Label\ro{\text {\small$X$}}(19,4)
\Label\ro{\text {\small$X$}}(22,7)
\Label\ro{\text {\small$X$}}(25,19)
\Label\ro{\text {\small$X$}}(28,1)
\Label\ro{\text {\small$X$}}(31,13)
\Label\ro{\text {\small$X$}}(34,19)
\Label\ro{\text {\small$X$}}(37,7)
\Label\ro{\text {\small$X$}}(40,22)
\Label\ro{\text {\small$X$}}(43,13)
\Label\ro{\text {\small$X$}}(46,1)
\Label\ro{\text {\small$X$}}(49,7)
\Label\ro{\text {\small$X$}}(52,16)
\Label\ro{\text {\small$r$}\rightarrow}(-3,1)
\hskip12cm
$$
\caption{The filling of $\mathcal{M}$ with $ne=5,\,se=5$.}
\label{fig:24}
\end{figure}

\begin{figure}[t]
$$
\Einheit.2cm
\PfadDicke{.5pt}
\Pfad(0,18),111111111111111111111111111111111111111111111111\endPfad
\Pfad(0,15),111111111111111111111111111111111111111111111111111111\endPfad
\Pfad(0,12),111111111111111111111111111111111111111111111111111111\endPfad
\Pfad(0,9),111111111111111111111111111111111111111111111111111111\endPfad
\Pfad(0,6),111111111111111111111111111111111111111111111111111111\endPfad
\Pfad(0,3),111111111111111111111111111111111111111111111111111111\endPfad
\Pfad(0,0),111111111111111111111111111111111111111111111111\endPfad
\Pfad(12,24),111111111111111111111111111111\endPfad
\Pfad(12,21),111111111111111111111111111111\endPfad
\Pfad(0,0),222222222222222222\endPfad
\Pfad(3,0),222222222222222222\endPfad
\Pfad(6,0),222222222222222222\endPfad
\Pfad(9,0),222222222222222222\endPfad
\Pfad(12,0),222222222222222222222222\endPfad
\Pfad(15,0),222222222222222222222222\endPfad
\Pfad(18,0),222222222222222222222222\endPfad
\Pfad(21,0),222222222222222222222222\endPfad
\Pfad(24,0),222222222222222222222222\endPfad
\Pfad(27,0),222222222222222222222222\endPfad
\Pfad(30,0),222222222222222222222222\endPfad
\Pfad(33,0),222222222222222222222222\endPfad
\Pfad(36,0),222222222222222222222222\endPfad
\Pfad(39,0),222222222222222222222222\endPfad
\Pfad(42,0),222222222222222222222222\endPfad
\Pfad(45,0),222222222222222222\endPfad
\Pfad(48,0),222222222222222222\endPfad
\Pfad(51,3),222222222222\endPfad
\Pfad(54,3),222222222222\endPfad
\PfadDicke{1pt}
\Pfad(0,18),111111111111111111111111111111111111111111111111\endPfad
\Pfad(0,15),111111111111111111111111111111111111111111111111\endPfad
\Pfad(0,15),222\endPfad
\Pfad(48,15),222\endPfad
\Label\ro{\text {\small$X$}}(1,16)
\Label\ro{\text {\small$X$}}(4,1)
\Label\ro{\text {\small$X$}}(7,7)
\Label\ro{\text {\small$X$}}(10,10)
\Label\ro{\text {\small$X$}}(13,7)
\Label\ro{\text {\small$X$}}(16,13)
\Label\ro{\text {\small$X$}}(19,4)
\Label\ro{\text {\small$X$}}(22,4)
\Label\ro{\text {\small$X$}}(25,19)
\Label\ro{\text {\small$X$}}(28,1)
\Label\ro{\text {\small$X$}}(31,10)
\Label\ro{\text {\small$X$}}(34,19)
\Label\ro{\text {\small$X$}}(37,4)
\Label\ro{\text {\small$X$}}(40,22)
\Label\ro{\text {\small$X$}}(43,10)
\Label\ro{\text {\small$X$}}(46,1)
\Label\ro{\text {\small$X$}}(49,4)
\Label\ro{\text {\small$X$}}(52,13)
\Label\ro{\text {\small$r^\prime$}\rightarrow}(-3,16)
\hskip12cm
$$
\caption{The filling of $\mathcal{M^\prime}$ with $ne=4,\,se=5$.}
\label{fig:25}
\end{figure}

\textbf{4.} The results of Rubey in~\cite{Rubey} for moon polyominoes were extended to almost-moon polyominoes by Poznanovi\'c and Yan~\cite{PY}. We say a row $r$ is an {\it exceptional row} of a polyomino $\mathcal{M}$ if there are rows above $r$ and below $r$ which are longer than $r$. Similarly, define the {\it exceptional column} $c$ of a polyomino $\mathcal{M}$ to be a column which has longer columns to both its right and its left.
An {\it almost-moon polyomino} is either a polyomino with comparable convex rows and at most one exceptional row or a polyomino with comparable convex columns and at most one exceptional column. See Figure~\ref{fig:23} for an illustration of almost-moon polyominoes with one exceptional row $r$ and one exceptional column $c$.

One may wonder whether the results in this paper could be extended to almost-moon polyominoes. The answer to this question is negative. We give two counterexamples below.

\begin{exa} \label{Ex7} 
Let $\mathcal{A}_{1}$ be the almost-moon polyomino with one exceptional row given in Figure~\ref{fig:23}a.
Let $G_{1}(x,y)=\sum_{M\in\mathcal{F}(\mathcal{A}_{1})} x^{\ne(M) } y^{\se(M)}$
be the joint distribution of $(ne,se)$ over fillings of $\mathcal{P}_{1}$. Then
\begin{align*}   
G_{1}(x,y)=(15x^5y^3+13x^3y^5)+56x^4y^4+(80x^5y^2+82x^2y^5)+(1180x^4y^3+1178x^3y^4)\nonumber \\ 
+5(x^5y+xy^5) +(1210x^4y^2+1212x^2y^4)+5370x^3y^3+10(x^4y+xy^4) \nonumber \\
+(1477x^3y^2+1473x^2y^3)+64x^2y^2.
\end{align*}
\end{exa}

\begin{exa} \label{Ex8} 
Let $\mathcal{A}_{2}$ be the almost-moon polyomino as given in Figure~\ref{fig:23}b.
Let $G_{2}(x,y)=\sum_{M\in\mathcal{F}(\mathcal{A}_{2})} x^{\ne(M) } y^{\se(M)}$
be the joint distribution of $(ne,se)$ over fillings of $\mathcal{P}_{2}$. Then
\begin{align*}   
G_{2}(x,y)=(8x^5y^3+15x^3y^5)+48x^4y^4+(83x^5y^2+77x^2y^5)+(1129x^4y^3+1174x^3y^4) \nonumber \\ 
+(9x^5y+8xy^5) +(1273x^4y^2+1227x^2y^4)+5434x^3y^3+(6x^4y+7xy^4) \nonumber \\
+(1415x^3y^2+1467x^2y^3)+60x^2y^2.
\end{align*}
\end{exa}

\begin{figure}[h]
$$
\Einheit.15cm
\PfadDicke{.5pt}
\Pfad(0,15),111111111111111111111\endPfad
\Pfad(0,12),111111111111111111111\endPfad
\Pfad(0,9),111111111111111111111\endPfad
\Pfad(0,6),111111111111111111111\endPfad
\Pfad(0,3),111111111111111111111\endPfad
\Pfad(0,0),111111111111111111111\endPfad
\Pfad(0,0),222222222\endPfad
\Pfad(0,12),222\endPfad
\Pfad(3,0),222222222222222\endPfad
\Pfad(6,0),222222222222222\endPfad
\Pfad(9,0),222222222222222\endPfad
\Pfad(12,0),222222222222222\endPfad
\Pfad(15,0),222222222222222\endPfad
\Pfad(18,0),222222222222222\endPfad
\Pfad(21,0),222222222222222\endPfad
\Label\ro{\text {\small$r$}\rightarrow}(-4,10)
\hbox{\hskip6cm}
\PfadDicke{.5pt}
\Pfad(0,15),111111\endPfad
\Pfad(9,15),111111111111\endPfad
\Pfad(0,12),111111111111111111111\endPfad
\Pfad(0,9),111111111111111111111\endPfad
\Pfad(0,6),111111111111111111111\endPfad
\Pfad(0,3),111111111111111111111\endPfad
\Pfad(0,0),111111111111111111111\endPfad
\Pfad(0,0),222222222222222\endPfad
\Pfad(3,0),222222222222222\endPfad
\Pfad(6,0),222222222222222\endPfad
\Pfad(9,0),222222222222222\endPfad
\Pfad(12,0),222222222222222\endPfad
\Pfad(15,0),222222222222222\endPfad
\Pfad(18,0),222222222222222\endPfad
\Pfad(21,0),222222222222222\endPfad
\Label\ro{\text {\small$c$}}(7,19)
\Label\ro{\text {$\downarrow$}}(7,16)
\hskip3cm
$$
\centerline{\small a. $\mathcal{A}_{1}$ 
\hskip5.5cm
b. $\mathcal{A}_{2}$}
\caption{Two almost-moon polyominoes.}
\label{fig:23}
\end{figure}

Note that if the generating function of $(\mathrm{ne}, \mathrm{se})$ is invariant under row or column permutations, then it is symmetric in $x$ and $y$. However, in   Examples~\ref{Ex7} and~\ref{Ex8} we have  $G_{1}(x,y) \neq G_{1}(y,x)$ and $G_{2}(x,y)\neq G_{2}(y,x)$.

\section*{Acknowledgments}
The authors thank Christian Krattenthaler
for many helpful discussions, also for his comments and corrections on a preliminary draft of the paper.
TG thanks the members of the combinatorics group at the University of Vienna for creating a wonderful, stimulating working environment and for their support during her stay in Vienna. SP also thanks the combinatorics group at the University of Vienna for their hospitality during her visit in the fall of 2018.

\bibliographystyle{plain}

\end{document}